\documentclass[a4paper]{amsart}
\usepackage{amssymb}
\usepackage{amsmath}
\usepackage{xypic}

\newtheorem{thm}{Theorem}[section]
\newtheorem*{thm*}{Theorem}

\newtheorem{cor}[thm]{Corollary}
\newtheorem*{cor*}{Corollary}
\newtheorem{lemma}[thm]{Lemma}
\newtheorem{prop}[thm]{Proposition}
\theoremstyle{definition}
\newtheorem{defn}[thm]{Definition}
\newtheorem{ex}[thm]{Example}

\theoremstyle{remark}
\newtheorem{rem}[thm]{Remark}

\newcommand{\ZZ}{{\mathbb  Z}}
\newcommand{\RR}{{\mathbb  R}}
\newcommand{\CC}{{\mathbb  C}}

\newcommand{\cV}{{\mathcal V}}

\newcommand{\E}{{\mathcal E}}

\newcommand{\F}{{\mathcal F}}

\newcommand{\so}{{{\mathfrak{so}}} }

\newcommand{\mfg}{\mathfrak{g}}
\newcommand{\mfh}{\mathfrak{h}}

\newcommand{\mfa}{\mathfrak{a}}

\newcommand{\mft}{\mathfrak{t}}

\newcommand{\Aut}{\operatorname{Aut}}

\newcommand{\depth}{\operatorname{depth}}

\newcommand{\Diff}{\operatorname{Diff}}

\newcommand{\odd}{\operatorname{odd}}

\newcommand{\Crit}{\operatorname{Crit}}

\newcommand{\SO}{\operatorname{SO}}

\newcommand{\Isom}{\operatorname{Isom}}
\newcommand{\Lie}{\operatorname{Lie}}

\newcommand{\I}{\operatorname{I}}
\newcommand{\V}{\operatorname{V}}
\newcommand{\Gr}{\operatorname{G}}

\numberwithin{equation}{section}

\begin{document}

\title{Equivariant cohomology of $K$-contact manifolds}

\author{Oliver Goertsches}
\address{Oliver Goertsches, Fachbereich Mathematik, Universit\"at Hamburg, Bundesstra\ss e 55 (Geomatikum), 20146 Hamburg, Germany}
\email{oliver.goertsches@math.uni-hamburg.de}
\author{Hiraku Nozawa}
\address{Hiraku Nozawa, Institut des Hautes \'{E}tudes Scientifiques, Le Bois-Marie 35, Route de Chartres 91440 Bures-sur-Yvette, France}
\email{nozawahiraku@06.alumni.u-tokyo.ac.jp}
\author{Dirk T\"oben}
\address{Dirk T\"oben, Instituto Matem\'atica e Estat\'istica, Universidade de S\~ao Paulo, Rua do Mat\~ao 1010, 05508-090 - S\~ao Paulo, Brazil}
\email{dtoeben@math.uni-koeln.de}

\date{}


\begin{abstract}
We investigate the equivariant cohomology of the natural torus action on a $K$-contact manifold and its relation to the topology of the Reeb flow. Using the contact moment map, we show that the equivariant cohomology of this action is Cohen-Macaulay, the natural substitute of equivariant formality for torus actions without fixed points. As a consequence, generic components of the contact moment map are perfect Morse-Bott functions for the basic cohomology of the orbit foliation $\F$ of the Reeb flow. Assuming that the closed Reeb orbits are isolated, we show that the basic cohomology of $\F$ vanishes in odd degrees, and that its dimension equals the number of closed Reeb orbits. We characterize $K$-contact manifolds with minimal number of closed Reeb orbits as real cohomology spheres. We also prove a GKM-type theorem for $K$-contact manifolds which allows to calculate the equivariant cohomology algebra under the nonisolated GKM condition.
\end{abstract}

\maketitle

\tableofcontents

\addtocontents{toc}{\protect\setcounter{tocdepth}{1}}
\section{Introduction} \label{sec-intro}

\subsection{Torus actions on $K$-contact manifolds}
A $K$-contact manifold is a manifold with a contact form $\alpha$ and a Riemannian metric $g$ adapted to $\alpha$ (see Definition~\ref{def:alphaadapted}). Any compact manifold $M$ with a contact form whose Reeb flow preserves a Riemannian metric on $M$ has a compatible $K$-contact structure by the observation of Yamazaki (Proposition~2.1 of~\cite{Yamazaki}). Main examples of $K$-contact manifolds are Sasakian manifolds, which have been studied by Einstein geometers and physicists (see Boyer-Galicki~\cite{Boyer Galicki}).

Let $(M,\alpha,g)$ be a compact connected $K$-contact manifold. The torus $T$ obtained by the closure of the Reeb flow of $\alpha$ in the isometry group of $(M,g)$ acts on $M$ without fixed points. The following theorem, which will be proven in Section~\ref{sec:main}, should be regarded as a result analogous to the statement (Proposition~5.8 of Kirwan~\cite{Kirwan}) that Hamiltonian actions on compact symplectic manifolds are equivariantly formal (see Section~\ref{sec:eqcohom}).
\begin{thm*}
The $T$-action on $M$ is Cohen-Macaulay.
\end{thm*}
\noindent See Definition~\ref{def:CMtorusactions} for the definition of Cohen-Macaulay actions. The $T$-action preserves the orbit foliation $\F$ of the Reeb flow. Therefore $(M,\F)$ admits a transverse action (see Definition~\ref{def:transverseaction}) of the abelian Lie algebra $\mfa\cong \mft/\RR R$ where $\mft = \Lie(T)$ and $R$ is the Reeb vector field of $\alpha$. In terms of basic equivariant cohomology of $(M,\F)$ introduced in \cite{GT2010}, the above theorem is equivalent to the following.
\begin{thm*}
The $\mfa$-action on $(M,\F)$ is equivariantly formal.
\end{thm*}
\noindent See Definition~\ref{def:eqformaltraction} for the definition of equivariant formality of transverse actions. Using results from~\cite{GT2009} and~\cite{GT2010}, these theorems will be deduced from the fact that a generic component of the contact moment map is a Morse-Bott function whose critical set equals the set of closed Reeb orbits.

\subsection{Topology of the Reeb flow of $K$-contact manifolds}
We will deduce several consequences on the topology of a compact $(2n+1)$-dimensional $K$-contact manifold $(M,\alpha,g)$ from the two theorems in the last section. Let $\F$ be the orbit foliation of the Reeb flow of $\alpha$, and $H^{\bullet}(M,\F)$ be the basic cohomology of $(M,\F)$. In particular, we will show in Section~\ref{sec:cons}
\begin{thm*} If the closed Reeb orbits of $\alpha$ are isolated, then we have
\begin{enumerate}
\item $H^1(M)=0$,
\item $H^{\odd}(M,\F)=0$ and
\item $\dim H^{\bullet}(M,\F) = \#\{$closed Reeb orbits of $\alpha \}$.
\end{enumerate}
\end{thm*}
\noindent If not further specified, cohomology is taken with real coefficients. Part $(1)$ of this theorem was originally proven by Rukimbira (Theorem~2 of~\cite{Rukimbira}, see also Theorem~7.4.8 of \cite{Boyer Galicki}). We remark that contact toric manifolds of Reeb type are examples of $K$-contact manifolds with isolated closed Reeb orbits (see Proposition~\ref{prop:toricisolatedorbits}).

Because the powers of the basic Euler class $d\alpha$ are nonzero elements of $H^\bullet(M,\F)$, part $(3)$ of the above theorem gives a new proof of the following result by Rukimbira (Corollary~1 of~\cite{Rukimbira}).

\begin{thm*} The Reeb flow of $\alpha$ has at least $n+1$ closed orbits.
\end{thm*}
The next theorem characterizes $K$-contact manifolds with minimal number of closed Reeb orbits.
\begin{thm*} If the closed Reeb orbits of $\alpha$ are isolated, then their number is exactly $n+1$ if and only if $M$ is a real cohomology sphere.
\end{thm*}
\noindent The only-if-part of this theorem improves a result of Rukimbira (Theorem~1 of~\cite{Rukimbira2}).

In Section~\ref{sec:Rukimbira} we provide an example of a $7$-dimensional simply-connec\-ted real cohomology sphere not homeomorphic to $S^7$ with a $K$-contact structure such that the closed Reeb orbits are isolated. This serves as a counterexample to Theorem~1 of Rukimbira~\cite{Rukimbira3} (quoted in Theorem~7.4.7 of \cite{Boyer Galicki}) which claims that a $K$-contact $(2n+1)$-manifold with exactly $n+1$ closed Reeb orbits is finitely covered by $S^{2n+1}$.

\subsection{GKM theory}
In Section~\ref{sec:GKM} we will show that a version of GKM theory~\cite{GKM} applies to Cohen-Macaulay actions and thus in particular to the above $T$-action on $K$-contact manifolds. This allows to compute the equivariant cohomology $H^*_T(M)$ as a graded $S(\mft^*)$-algebra. These results show that as in the case of Hamiltonian actions on symplectic manifolds, there is a strong link between the topology of $K$-contact manifolds and their equivariant cohomology.

\noindent {\bf Acknowledgements.} The second author is partially supported by Postdoctoral Fellowship of the French Government (662014L) and the Research Fellowship of Canon Foundation in Europe. Tis work started after the second and third authors met at Centre de Recerca Matem\`{a}tica in Barcelona during the research programme "Foliations" in the spring of 2010. The second and the third authors are grateful to Centre de Recerca Matem\`{a}tica for the invitation and the hospitality.

\section{$K$-contact manifolds} \label{sec:K-contact}

\subsection{Fundamentals}
Let $M$ be an odd-dimensional compact manifold with a contact form $\alpha$ and a Riemannian metric $g$. The Reeb vector field $R$ of $\alpha$ is characterized by the conditions $\alpha(R)=1$ and $\iota_R d\alpha=0$. The {\em Reeb flow} of $\alpha$ is the flow generated by $R$. The Reeb flow of $\alpha$ leaves $\alpha$ invariant.
\begin{defn}\label{def:alphaadapted}
A Riemannian metric $g$ on $(M,\alpha)$ is said to be {\em adapt\-ed to} $\alpha$ if
\begin{enumerate}
\item $g$ is preserved by the Reeb flow of $\alpha$ and
\item there exists an almost complex structure $J$ on $\ker \alpha$ such that we have $g(X,Y) = d\alpha(X,JY)$ for all $X$, $Y$ in $C^{\infty}(\ker \alpha)$.
\end{enumerate}
\end{defn}
\noindent We recall
\begin{defn}
$(M,\alpha,g)$ is called a {\em $K$-contact manifold} if $g$ is adapt\-ed to $\alpha$.
\end{defn}
A compact $K$-contact manifold has a natural torus action: The closure $T$ of the Reeb flow in the isometry group $\Isom(M,g)$ of $(M,g)$ is a connected abelian Lie subgroup. As  $\Isom(M,g)$ is compact by the Myers-Steenrod theorem~\cite{Myers Steenrod}, so is $T$, which implies that $T$ is a torus.
\begin{defn}
The {\em rank} of $(M,\alpha,g)$ is the dimension of $T$.
\end{defn}
\noindent The rank of a compact $K$-contact $(2n+1)$-manifold $(M,\alpha,g)$ is bounded from above by $n+1$ (see Corollary~1 of Rukimbira~\cite{Rukimbira4}). Because the Reeb flow preserves $\alpha$, the $T$-action preserves $\alpha$ by continuity. Note that $T$ as a subtorus in the diffeomorphism group of $M$ is independent of the choice of $g$. The following observation due to Yamazaki characterizes contact forms which have a compatible $K$-contact structure in terms of torus actions:
\begin{prop}[{\cite[Proposition~2.1]{Yamazaki}}]\label{prop:Yamazaki} The following are equivalent for a compact manifold $M$ with a contact form $\alpha$:
\begin{enumerate}
\item $(M,\alpha)$ admits a metric adapted to $\alpha$,
\item the Reeb flow of $\alpha$ preserves a Riemannian metric on $M$ and
\item there exists a torus action on $M$ such that the Reeb flow of $\alpha$ is a dense subaction of the torus action.
\end{enumerate}
\end{prop}

\begin{ex}
Main examples of $K$-contact manifolds are Sasakian manifolds. They include contact toric manifolds of Reeb type (see below for the definition and see Theorem~5.2 of Boyer-Galicki~\cite{Boyer Galicki 2} for the existence of a Sasakian metric) and links of isolated singularities of weighted homogeneous polynomials (see Chapters~7~and~9 of \cite{Boyer Galicki} and references therein).
\end{ex}

\begin{defn}
A $(2n+1)$-dimensional contact manifold with a $T^{n+1}$-action which preserves the contact structure is called a {\em contact toric manifold} (see Lerman~\cite{Lerman}). Moreover, if the Reeb vector field of a contact form generates an $\RR$-subaction of the $T^{n+1}$-action, then the contact $T^{n+1}$-manifold is called a {\em contact toric manifold of Reeb type} (see \cite{Boyer Galicki 2}).
\end{defn}

\subsection{Deformation of Reeb vector fields}
For later use, let us state some results, Proposition~1 of Banyaga-Rukim\-bira~\cite{Banyaga Rukimbira} and Lemma~2.5 of \cite{Nozawa}, that enable us to modify a $K$-contact structure by deforming its Reeb vector field in some torus. These are $K$-contact variants of Theorem~A of Takahashi~\cite{Takahashi} for Sasakian manifolds (note that $K$-contact manifolds are defined in a different and nonstandard way in \cite{Nozawa}).
\begin{lemma}\label{lem:modification}
Let $(M,\alpha,g)$ be a compact $K$-contact manifold.
\begin{enumerate}
\item  Let $T'$ be an $S^{1}$-subgroup of $T$ whose orbits are transverse to $\ker \alpha$. Then $M$ has a $K$-contact structure $(\alpha',g')$ such that the closure of the Reeb flow of $\alpha'$ in $\Isom(M,g')$ is equal to $T'$.
\item  Let $T'$ be a torus in $\Diff(M)$ which contains $T$ and acts on $M$ preserving $\alpha$. Then $M$ has a $K$-contact structure $(\alpha',g')$ such that the closure of the Reeb flow of $\alpha'$ in $\Isom(M,g')$ is $T'$.
\end{enumerate}
\end{lemma}

\begin{proof}
Let us prove (1). Let $Z$ be a vector field on $M$ which generates the $T'$-action. The transversality of $Z$ to $\ker \alpha$ implies that $\alpha(Z)$ is a nowhere vanishing function. Then $\alpha'=\frac{1}{\alpha(Z)}\alpha$ is a contact form whose Reeb vector field is $Z$. Then (1) follows from Proposition~\ref{prop:Yamazaki}.

Let us prove (2). We take a fundamental vector field $Z$ of the $T'$-action sufficiently close to the Reeb vector field of $\alpha$ so that the flow generated by $Z$ is dense in $T'$. By compactness of $M$, the orbits of the flow generated by $Z$ are transverse to $\ker \alpha$. Letting $\alpha'=\frac{1}{\alpha(Z)}\alpha$, we get a contact form $\alpha'$ whose Reeb vector field is $Z$. Then Proposition~\ref{prop:Yamazaki} concludes the proof of (2). 
 \end{proof}

\begin{rem}\label{rem:transverse}
In Lemma~\ref{lem:modification} (1), if $Z$ is sufficiently close to the Reeb vector field, then the orbits of the $T'$-action are transverse to $\ker \alpha$ by compactness of $M$.
\end{rem}

\section{Transverse actions on foliated manifolds}

In this section, let us recall the definition of transverse actions on foliated manifolds introduced in \cite{GT2010}.

Let $\F$ be a foliation of a manifold $M$. By $\Xi(\F)$ we denote the space of differentiable vector fields on $M$ that are tangent to the leaves of $\F$. A vector field $X$ on $M$ is said to be {\em foliate} if for every $Y\in \Xi(\F)$ the Lie bracket $[X,Y]$ also belongs to $\Xi(\F)$. A vector field is foliate if and only if its flow maps leaves of $\F$ to leaves of $\F$, see Proposition~2.2 of Molino~\cite{Molino}. The set $L(M,\F)$ of foliate fields is the normalizer of $\Xi(\F)$ in the Lie algebra $\Xi(M)$ of vector fields on $M$ and therefore a Lie sub-algebra of $\Xi(M)$. We call the projection of a foliate field $X$ to $TM/T\F$ a {\em transverse} field. The set $l(M,\F)=L(M,\F)/\Xi(\F)$ of transverse fields is also a Lie algebra inheriting the Lie bracket from $L(M,\F)$.
\begin{defn}[{\cite[Definition 2.1]{GT2010}}]\label{def:transverseaction}
A {\em transverse action} on the foliated manifold $(M,\F)$ of a fi\-nite-dimensional Lie algebra $\mfg$ is a Lie algebra homomorphism $\mfg\to l(M,\F)$.
\end{defn}
\noindent Given a transverse action of $\mfg$, we will denote the transverse field associated to $X\in \mfg$ by $X^{\#}\in l(M,\F)$.
If $\F$ is the trivial foliation by points, this notion coincides with the usual notion of an infinitesimal action on the manifold $M$.

Let us now return to $K$-contact manifolds. Denote by $\F$ the orbit foliation of the Reeb flow on a $K$-contact manifold $(M,\alpha,g)$. This is the Riemannian foliation defined by the isometric flow generated by $R$ (see Carri\`{e}re~\cite{Carriere}). By the commutativity of $T$, the $T$-action preserves $\F$. Hence there is a canonical map
\begin{equation}\label{eq:traction}
\mfa := \mft/\RR R \longrightarrow l(M,\F).
\end{equation}
This defines a transverse action of the abelian Lie algebra $\mfa$ on $(M,\F)$.

Because $\F$ is the orbit foliation of an isometric flow, its Molino sheaf is trivial by Th\'{e}or\`{e}me~A of Molino-Sergiescu~\cite{MolSer} (i.e., $\F$ is a Killing foliation), which we may identify with $\mfa$ by Example~4.3 of~\cite{GT2010}. Note that in~\cite{GT2010} the Lie algebra $\mfa$ is called the structural Killing algebra of $\F$.

\section{Equivariant cohomology} \label{sec:eqcohom}

\subsection{$\mfg$-differential graded algebras and the Cartan model}

In this section, we recall the Cartan model of equivariant cohomology in the language of differential graded algebras.

\begin{defn}\label{defn:g*}
Let $\mfg$ be a finite-dimensional Lie algebra and $A=\bigoplus A_k$ a $\ZZ$-graded algebra. We call $A$ a $\mfg$-{\em differential graded algebra} ($\mfg$-{\em dga}) if there is a derivation $d:A\to A$ of degree $1$ and  derivations $\iota_X:A\to A$ of degree $-1$ and $L_X:A\to A$ of degree $0$ for all $X\in \mfg$ (where $\iota_X$ and $L_X$ depend linearly on $X$) such that:
\begin{enumerate}
\item $d^2=0$
\item $\iota_X^2=0$
\item $[L_X,L_Y]=L_{[X,Y]}$
\item $[L_X,\iota_Y]=\iota_{[X,Y]}$
\item $[d,L_X]=0$
\item $L_X=d\iota_X+\iota_Xd$.
\end{enumerate}
\end{defn}

\begin{ex} \label{ex:gstaraction}
An infinitesimal action of a finite-dimensional Lie algebra $\mfg$ on a manifold $M$, i.e.~a Lie algebra homomorphism $\mfg\to \Xi(M);\, X\mapsto X^{\#}$, induces a $\mfg$-dga structure on the de Rham complex $\Omega(M)$ with operators $\iota_X:=\iota_{X^{\#}}$ and $L_X:=L_{X^{\#}}$.
\end{ex}

The {\em Cartan complex} of $A$ is defined as
\begin{equation*}
C_\mfg(A):=(S(\mfg^*)\otimes A)^\mfg.
\end{equation*}
\noindent Here the superscript denotes the subspace of $\mfg$-invariant elements, i.e., those $\omega\in S(\mfg^*)\otimes A$ for which $L_X\omega=0$ for all $X\in \mfg$. The differential $d_\mfg$ of the Cartan complex $C_\mfg(A)$ is defined by
$$
(d_\mfg \omega)(X)=d(\omega(X))-\iota_X(\omega(X)),
$$
\noindent where we consider an element in $C_\mfg(A)$ as a $\mfg$-equivariant polynomial map $\mfg\to A$.
Now the equivariant cohomology of the $\mfg$-dga $A$ is defined as
\[
H_\mfg^\bullet(A):=H^\bullet(C_\mfg(A),d_\mfg).
\]
\noindent There is a natural $S(\mfg^*)^\mfg$-algebra structure on $H^\bullet_\mfg(A)$.

\subsection{Equivariant cohomology of Lie group actions}

Let $G$ be a compact connected Lie group acting on a manifold $M$. Let $\mfg$ be the Lie algebra of right-invariant vector fields on $G$. Then the $G$-action induces a Lie algebra homomorphism $\mfg\to \Xi(M)$, and hence the structure of a $\mfg$-dga on the algebra of differential forms $\Omega(M)$, see Example~\ref{ex:gstaraction}. The equivariant cohomology of the $G$-action is
\[
H^\bullet_G(M)=H_\mfg^\bullet(\Omega(M)).
\]
\noindent This so-called Cartan model of equivariant cohomology is isomorphic to the Borel model (which is defined for more general spaces than manifolds)
\[
H^\bullet_G(M)=H^\bullet(EG\times_G M),
\]
where $EG$ is a contractible space on which G acts freely. Unless otherwise stated, cohomology is taken with real coefficients.

The $G$-action is called {\em equivariantly formal} if
\[
H^\bullet_G(M)\cong S(\mfg^*)^G \otimes H^\bullet(M)
\]
\noindent as graded $S(\mfg^*)^G$-modules.

\subsection{Equivariant basic cohomology}\label{sec:transverseactions}

Let $\F$ be a foliation on a manifold $M$. A differential form $\omega\in \Omega(M)$ is {\em $\F$-basic} if $\iota_X\omega=0$ and $L_X\omega=0$ for all $X\in \Xi(\F)$. We denote the algebra of basic differential forms by $\Omega(M,\F)$. For $X\in l(M,\F)$, $\iota_X$ and $L_X$ are well-defined derivations on $\Omega(M,\F)$.

\begin{prop}[{\cite[Proposition~3.12]{GT2010}}]
A transverse action of a finite-dimen\-sional  Lie algebra $\mfg$ on a foliated manifold $(M,\F)$ induces the structure of a $\mfg$-dga on $\Omega(M,\F)$.
\end{prop}

\begin{defn}[{\cite[Section~3.6]{GT2010}}]\label{def:eqformaltraction}
The equivariant basic cohomology of a transverse $\mfg$-action on $(M,\F)$ is defined as
\[
H^\bullet_\mfg(M,\F)=H^\bullet_\mfg(\Omega(M,\F)).
\]
\noindent The $\mfg$-action is called {\em equivariantly formal} if
 \[
H^\bullet_\mfg(M,\F)\cong S(\mfg^*)^\mfg\otimes H^\bullet(M,\F)
\]
\noindent as graded $S(\mfg^*)^\mfg$-modules.
\end{defn}
\noindent The notion of equivariant formality of transverse actions is analogous to that of Lie group actions.

\section{Equivariantly formal and Cohen-Macaulay actions}\label{sec:eqformalandCM}

It is proven in Proposition~5.1 of Franz-Puppe~\cite{FrPu} (see also Definition~(4.1.5) of Allday-Puppe~\cite{allday}) that for any action of a torus $T$ on a compact manifold $M$ the (Krull) dimension of the $S(\mft^*)$-module $H^{\bullet}_T(M)$ is equal to the dimension of a maximal isotropy algebra (i.e., the Lie algebra of an isotropy group).

\begin{defn}\label{def:CMtorusactions}
The $T$-action is said to be {\it Cohen-Macaulay} if $H^{\bullet}_T(M)$ is a Cohen-Macaulay module over $S(\mft^*)$, i.e., if $\dim_{S(\mft^*)}H^{\bullet}_T(M)$ equals $\depth_{S(\mft^*)}H^{\bullet}_T(M)$.
\end{defn}
\noindent The depth of a finitely generated graded $S(\mft^*)$-module is always bounded from above by its dimension.
\begin{ex} \label{ex:locallyfree} Any locally free torus action is Cohen-Macaulay. In fact, the dimension of a maximal isotropy algebra is zero, which forces the depth of $H^{\bullet}_T(M)$ to be zero as well.
\end{ex}
If the $T$-action is equivariantly formal, then $H^\bullet_T(M)$ is a free $S(\mft^*)$-module and in particular Cohen-Macaulay. In fact, Cohen-Macaulay actions were introduced in~\cite{GT2009} as a generalization of equivariantly formal actions. More precisely, we have

\begin{prop}[{\cite[Proposition~6.2]{GT2009}}]\label{prop:formaleqandCM} The $T$-action is equivariantly formal if and only if it is Cohen-Macaulay and has fixed points.
\end{prop}

\begin{prop}[{\cite[Remark~2]{GT2009}}]\label{prop:commutingaction} Let $T'\subset T$ be a subtorus maximal among those subtori acting on $M$ locally freely. Then $H^\bullet_T(M)\cong H^\bullet_{T/T'}(M/T')$ as graded rings. The $T$-action on $M$ is Cohen-Macaulay if and only if the $T/T'$-action on $M/T'$ is equivariantly formal.
\end{prop}

\section{The contact moment map and equivariant cohomology} \label{sec:main}
Let $(M,\alpha,g)$ be a compact connected $K$-contact manifold with Reeb vector field $R$. Let $T$ be the closure of the Reeb flow in $\Isom (M,g)$. Then $T$ is a torus with $R\in \mft = \Lie(T)$. Denote by $C$ the union of the closed Reeb orbits. Clearly, $C$ coincides with the union of all one-dimensional $T$-orbits. We recall

\begin{defn}
The {\em contact moment map} of $(M,\alpha,g)$ is given by
\[
\Phi:M\to \mft^*;\, \Phi(p)(X)=\alpha_p(X^{\#}_p).
\]
\noindent For each $X\in \mft$, we may consider the $X$-component $\Phi^X(p)=\Phi(p)(X)$ of $\Phi$.
\end{defn}

\begin{rem} The $T$-equivariant differential $d_T\alpha\in C^2_\mft(\Omega(M))$ of the contact form $\alpha$ is  given by
\[
(d_T\alpha)(X)=d\alpha - \iota_{X^{\#}}\alpha
\]
\noindent  where $X\in \mft$. From this point of view, the (negative of the) contact moment map $\Phi(X)=\iota_{X^{\#}}\alpha$ can be viewed as the $\big(S(\mft^*)\otimes C^{\infty}(M)\big)$-part of an equivariantly exact extension of the exact form $d\alpha$. This is comparable to the case of a Hamiltonian action on a symplectic manifold, in which the moment map is the $\big(S(\mft^*)\otimes C^{\infty}(M)\big)$-part of an equivariantly closed extension of the symplectic form, see for example Proposition~VI.2.1 of Audin~\cite{Audin}.
\end{rem}

\begin{prop} \label{prop:momentmorsebott} For generic $X$ in $\mft$ (as defined in the proof), the function $\Phi^X$ is a $T$-invariant Morse-Bott function with critical set $C$.
\end{prop}
\noindent This proposition is well-known, see for example Section~2 of~\cite{Rukimbira2}. For the convenience of the reader we include a proof that the critical set is exactly $C$. For the nondegeneracy of the Hessian in normal directions, see Lemma~1.(2) of~\cite{Rukimbira2}.
\begin{proof}
If the rank of $(M,\alpha)$ is one, then $\mft$ is spanned by the Reeb vector field, and the function $\Phi^R$ is constant. Because in this case $C=M$ the statement follows trivially for all $X$ in $\mft$ (and we call any element $X$ in $\mft$ generic).

Assume that the rank of $(M,\alpha)$ is greater than one.  For each $p\in M$, let $\tilde{\mft}_p:=\{X\in \mft\mid X^{\#}_p\in \RR R_p\}$. Clearly, $\tilde{\mft}_p=\mft_p\oplus \RR R$, so there are only finitely many distinct subspaces $\tilde{\mft}_p\subset \mft$. We have $\tilde{\mft}_p=\mft$ if and only if $p$ is contained in a closed Reeb orbit. We say that $X$ in $\mft$ is generic if
\[
X\notin \bigcup_{p:\tilde{\mft}_p\neq \mft} \tilde{\mft}_p.
\]
\noindent We claim that if $X$ is generic, then the critical set of $\Phi^X$  is  $C$. The following argument is due to Lemma~2.1 of \cite{Banyaga Rukimbira}. We have
\[
(d\Phi^X)_p(v)=-(d\alpha)_p(X^{\#}_p,v)
\]
\noindent for all $v\in T_pM$. Because $(d\alpha)_p$ is nondegenerate on $T_pM/\RR R_p$, the critical set of $\Phi^X$ is equal to $\{p \in M \mid X^{\#}_p\in \RR R_p\}$. By definition of $X$, this set is $C$.
 \end{proof}

\begin{cor}[Banyaga~\cite{Banyaga}, Rukimbira~\cite{Rukimbira}]\label{cor:C}
The Reeb flow of $\alpha$ has closed orbits, i.e., $C\neq \emptyset$.
\end{cor}

In the following we apply the theory presented in \cite{GT2009} and~\cite{GT2010} to Proposition~\ref{prop:momentmorsebott}.

\begin{thm}\label{thm:actionisCM} The $T$-action on $M$ is Cohen-Macaulay.
\end{thm}
\begin{proof} If the rank $(M,\alpha)$ is one, then $T$ is a circle acting locally freely, hence the action is Cohen-Macaulay by Example~\ref{ex:locallyfree}.
The general case follows directly from Theorem~7.1 of~\cite{GT2009} and Proposition~\ref{prop:momentmorsebott}.
 \end{proof}
\begin{rem}\label{rem:CM}
Let $S^1\subset T$ be a circle whose orbits are transverse to $\ker \alpha$. By Lemma~\ref{lem:modification} (1), we have a $K$-contact structure $(\alpha',g')$ whose Reeb flow is given by $S^1$. The two-form $d\alpha'$ is $S^1$-basic and thus descends to the orbifold $M/S^1$ thereby turning it into a symplectic orbifold. The contact moment map $\Phi:M\to \mft^*;\, p\mapsto \alpha'_p$ descends to a moment map for the $T/S^1$-action on $M/S^1$. Then, by Proposition~\ref{prop:commutingaction} the $T$-action on $M$ is Cohen-Macaulay if and only if the $T/S^1$-action on the orbifold $M/S^1$ is equivariantly formal. Therefore, if there existed an orbifold version of the classical result of Kirwan (Proposition~5.8 of~\cite{Kirwan}) that Hamiltonian actions on compact symplectic manifolds are equivariantly formal, it would also give a proof of Theorem~\ref{thm:actionisCM}.
\end{rem}
Lemma~\ref{lem:modification} (2) and Theorem~\ref{thm:actionisCM} imply:
\begin{cor} Any $\alpha$-preserving effective torus action on $M$ that has $R$ as a fundamental vector field is Cohen-Macaulay.
\end{cor}

Let $\F$ be the orbit foliation of the Reeb flow, which is Riemannian. We consider the transverse action of $\mfa=\mft/\RR R$ on the foliated manifold $(M,\F)$ (see Definition~\ref{def:transverseaction} and Equation \eqref{eq:traction}).

\begin{thm} \label{thm:foliationeqformal} The $\mfa$-action on $(M,\F)$ is equivariantly formal, i.e., we have $H^{\bullet}_\mfa(M,\F)\cong S(\mfa^*)\otimes H^{\bullet}(M,\F)$ as graded $S(\mfa^{*})$-modules.
\end{thm}
\begin{proof} This follows from Theorem~6.3 of~\cite{GT2010} and Proposition~\ref{prop:momentmorsebott}.
\end{proof}

\begin{rem}
In view of the equality
\begin{equation}\label{eq:gradedrings}
H^{\bullet}_T(M)\cong H^{\bullet}_\mfa(M,\F)
\end{equation}
\noindent as graded rings shown in Example 4.3 of~\cite{GT2010}, the statements of Theorems~\ref{thm:actionisCM} and~\ref{thm:foliationeqformal} are equivalent. In fact, because the Riemannian foliation $\F$ has closed leaves by Corollary~\ref{cor:C}, the Cohen-Macaulay property of $H^{\bullet}_T(M)$ is equivalent to equivariant formality of the $\mfa$-action, see Example 7.5 of~\cite{GT2010}.
\end{rem}

\begin{prop}[{\cite[Theorem~6.4]{GT2010}}] \label{prop:perfectbasic}
For generic $X$ in $\mft$, the function $\Phi^X$ is a perfect basic Morse-Bott function (with critical set $C$). More precisely: The basic Poincar\'e series $P_t(M,\F)$ given by $P_t(M,\F)=\sum_j t^j \dim H^j(M,\F)$ can be calculated as
\[
P_t(M,\F)=\sum_{B} t^{\lambda_{B}} P_t(B/\F),
\]
\noindent where $B$ runs over the connected components of $C$, the leaf space of $(B,\F|_{B})$ is denoted by $B/\F$ and $\lambda_{B}$ is the index of $\Phi^X$ at $B$.
\end{prop}

\begin{rem}
Note that $\Phi^X$ is also $T$-equivariantly perfect and $\mfa$-equivari\-antly basic perfect, but not necessarily perfect in the ordinary sense. For example, take the $K$-contact structure of rank $2$ on $S^3$ considered in Example~4.4 of~\cite{GT2010}, which is obtained by deforming the standard $K$-contact structure on $S^{3}$ using Lemma \ref{lem:modification}. The two critical manifolds are circles, so $\Phi^X$ cannot be perfect.
\end{rem}
\begin{rem}
In view of Proposition~\ref{cor:circleaction} below, Proposition~\ref{prop:perfectbasic} can be proven for isolated closed Reeb orbits using Lerman-Tolman \cite{LerTol}, Re\-mark~5.4.
\end{rem}

\section{Implications for the topology of the Reeb flow}\label{sec:cons}

We will apply Theorems~\ref{thm:actionisCM},~\ref{thm:foliationeqformal} and Proposition~\ref{prop:perfectbasic} to obtain various consequences on the topology of $K$-contact manifolds, in particular in the case where the closed Reeb orbits are isolated. First, we present well-known examples of $K$-contact manifolds whose closed Reeb orbits are isolated. For a representation $V$ of a $k$-dimensional torus $T^{k}$, let $V_{\mu}$ be the weight space of $V$ corresponding to a weight $\mu$. If $\mu_i$ are the weights of a $T$-representation $V$, then $V \cong \bigoplus_{i} V_{\mu_{i}}$. The following is well known:

\begin{lemma}\label{lem:weight}
Any faithful $T^{k}$-representation $V$ of dimension $2k$ has exactly $k$  weights $\mu_1$, $\ldots$, $\mu_k$. They are linearly independent and their weight spaces are two-dimensional.
\end{lemma}

\begin{prop}\label{prop:toricrem}
Every closed Reeb orbit of a compact $K$-contact $(2n+1)$-manifold $(M,\alpha,g)$ of rank $n+1$ is isolated.
\end{prop}
\begin{proof}
Let $x$ be a point in a closed Reeb orbit of $\alpha$. The identity component of the isotropy group $T_{x}$ at $x$ effectively acts on $(\ker \alpha)_{x}$. By Lemma~\ref{lem:weight}, the fixed point of the $T_{x}$-action on $(\ker \alpha)_{x}$ is only $0$. This implies that the closed Reeb orbit of $x$ is isolated.
 \end{proof}

\noindent Lemma~\ref{lem:modification}~(2) and Proposition~\ref{prop:toricrem} imply the following well-known result:
\begin{prop}\label{prop:toricisolatedorbits}
Any compact toric contact $(2n+1)$-manifold of Reeb type has a $K$-contact structure of rank $n+1$ such that every closed Reeb orbit is isolated.
\end{prop}

Below we will use the notation from Section~\ref{sec:main}. The following is one of the main results in this paper:
\begin{thm}\label{cor:odd} Assume that the closed Reeb orbits are isolated. Then we have $H^{\odd}(M,\F)=0$ and consequently, $H^1(M)=0$. Also, $H^{\odd}_{\mfa}(M,\F)=0$.
\end{thm}
\begin{proof} For each connected component $B$ of the union $C$ of closed Reeb orbits, we have $P_t(B/\F)=1$. Moreover, the unstable normal bundles are invariant under the action of isotropy groups and hence of even rank. Then $H^{\odd}(M,\F)=0$ follows from Propositions~\ref{prop:momentmorsebott} and~\ref{prop:perfectbasic}. By the Gysin sequence for an isometric flow $\F$, see Saralegui~\cite{Saralegui} or Equation~(7.2.1) in p.~215 of \cite{Boyer Galicki}, we have $H^{1}(M,\F) \cong H^{1}(M)$, which proves the second assertion. $H^{\odd}(M,\F)=0$ and Theorem~\ref{thm:foliationeqformal} imply the last assertion.
 \end{proof}

\begin{rem}
In Theorem~2 of~\cite{Rukimbira}, Rukimbira showed that $H^{1}(M)=0$ if the closed Reeb orbits are isolated and $M$ is a Sasakian manifold. As stated in Theorem~7.4.8 of~\cite{Boyer Galicki}, his proof can be extended to show $H^{1}(M)=0$ under the assumption of Theorem~\ref{cor:odd} in the $K$-contact case.
\end{rem}

\begin{rem}
Theorem~\ref{cor:odd} follows also from Theorem~\ref{thm:foliationeqformal} and the Borel-type Localization Theorem for equivariant basic cohomology (Theorem 5.2 of~\cite{GT2010}): Because the $\mfa$-action is equivariantly formal by Theorem~\ref{thm:foliationeqformal}, $H^\bullet_\mfa(M,\F)$ is a torsion-free $S(\mfa^*)$-module, and so the natural restriction map $H^{\bullet}_\mfa(M,\F)\to H^{\bullet}_\mfa(C,\F)\cong \bigoplus_{B} S(\mfa^*)$ is injective, where $B$ runs over the connected components of $C$. The right hand side has no elements of odd degree, which implies $H^{\odd}_\mfa(M,\F)=0$ and hence by Theorem~\ref{thm:foliationeqformal} also $H^{\odd}(M,\F)=0$.
\end{rem}

\begin{cor}\label{cor:S1} If $S^1\subset T$ is any circle acting  locally freely on $M$, then the $T/S^1$-action on the orbifold $M/S^1$ is equivariantly formal.
\end{cor}
\begin{proof}
This follows directly from Theorem~\ref{thm:actionisCM} and Proposition~\ref{prop:commutingaction} because the $T/S^1$-action has fixed points.
 \end{proof}

\begin{rem}
When the orbits of the $S^{1}$-action are transverse to the contact structure $\ker \alpha$, the $S^{1}$-action is the Reeb flow of a $K$-contact structure by Lemma~\ref{lem:modification}~(1). In this case, the $T/S^{1}$-action on $M/S^{1}$ is a Hamiltonian torus action on a symplectic orbifold. Thus Corollary~\ref{cor:S1} provides examples of equivariant formal Hamiltonian torus actions on symplectic orbifolds.
\end{rem}

\begin{prop}\label{cor:circleaction} If $S^1\subset T$ is any circle acting locally freely on $M$, then $H^{\bullet}(M,\F)$ and $H^{\bullet}(M/S^1)$ are isomorphic as graded vector spaces.
\end{prop}
\begin{proof} We have isomorphisms
\begin{equation}\label{eq:isomgradedrings}
H^\bullet_\mfa(M,\F)\cong H^\bullet_T(M)\cong H^\bullet_{T/S^1}(M/S^1)
\end{equation}
\noindent as graded rings, see \eqref{eq:gradedrings} and Proposition~\ref{prop:commutingaction}. Furthermore, both the transverse action of $\mfa$ as well as the $T/S^1$-action on $M/S^1$ are equivariantly formal by Theorem~\ref{thm:foliationeqformal} and Corollary~\ref{cor:S1}, hence the Poincar\'{e} series of the graded vector spaces in \eqref{eq:isomgradedrings} are given by
\begin{align*}
P_t(S(\mfa^*))\cdot P_t(M,\F)&=P_t^\mfa(M,\F)\\
&=P_t^{T/S^1}(M/S^1)=P_t(S((\mft/\RR)^*))\cdot P_t(M/S^1).
\end{align*}
\noindent Because $\mfa$ and $\mft/\RR$ are vector spaces of the same dimension, this implies that the Poincar\'{e} polynomials $P_t(M,\F)$ and $P_t(M/S^1)$ coincide.
 \end{proof}

\begin{ex}
Let $(X,\omega)$ be a closed symplectic manifold such that $[\omega]\in H^{2}(X)$ is integral. By the Boothby-Wang construction~\cite{Boothby Wang}, we have a principal $S^{1}$-bundle $\pi \colon M \to X$ with a contact form $\alpha$ which satisfies $d\alpha=\pi^{*}\omega$ and whose Reeb flow is the principal $S^{1}$-action. By Proposition~\ref{prop:Yamazaki}, we have a $K$-contact structure $(\alpha,g)$ on $M$. For any $K$-contact structure $(\alpha',g')$ on $M$ obtained by applying Lemma~\ref{lem:modification} (2) to $(\alpha,g)$, the basic cohomology  of the orbit foliation of the Reeb flow of $\alpha'$ is isomorphic to $H^{\bullet}(X)$ by Proposition~\ref{cor:circleaction}. Thus Theorem~\ref{cor:odd} implies that if $H^{\odd}(X)$ is nontrivial, then we cannot deform the Reeb vector field of $\alpha$ by Lemma~\ref{lem:modification} (2) to obtain a $K$-contact structure on $M$ whose Reeb flow has only isolated closed Reeb orbits.
\end{ex}

\begin{thm}\label{cor:numC} We have $\dim H^{\bullet}(C/\F)=\dim H^{\bullet}(M,\F)$. In particular, in case the closed Reeb orbits are isolated, their number is given by $\dim H^{\bullet}(M,\F)$.
\end{thm}
\begin{proof} This is Theorem~5.5 of~\cite{GT2010} (which is a consequence of Theorem~5.2 of~\cite{GT2010}, the Borel-type Localization Theorem for equivariant basic cohomology) as we have proven in Theorem~\ref{thm:foliationeqformal} that the $\mfa$-action on $(M,\F)$ is equivariantly formal. Alternatively, one may use Proposition~\ref{cor:circleaction} and apply Corollary~2 in p.~46 of Hsiang~\cite{Hsiang} because $T/S^1$ acts equivariantly formally on $M/S^1$ by Corollary~\ref{cor:S1}, and the $T/S^1$-fixed point set is exactly $C/\F$.
 \end{proof}
\noindent In the case where $(M,\alpha,g)$ is a contact toric manifold of Reeb type, Corollary~\ref{cor:numC} implies that $\dim H^{\bullet}(M,\F)$ is equal to the number of vertices of the image of the contact moment map (see Example~\ref{ex:toricGKM}).

In the next corollary we give a new proof of a result by Rukimbira. Let $\dim M = 2n+1$.

\begin{cor}[{\cite{Rukimbira}}]\label{cor:Rukimbira} The Reeb flow of $\alpha$ has at least $n+1$ closed orbits.
\end{cor}

\begin{proof}
We may assume that the closed Reeb orbits are isolated, as otherwise there exist infinitely many. Let $[d\alpha] \in H^{2}(M,\F)$ be the basic Euler class of $(M,\F)$ (see~\cite{Saralegui}). Since $\alpha$ is a contact form, $[(d\alpha)^{k}]$ is nontrivial in $H^{2k}(M,\F)$ for $1 \leq k \leq n$. Thus we have $\dim H^{\bullet}(M,\F) \geq n + 1$. Then Theorem~\ref{cor:numC} concludes the proof.
 \end{proof}

\begin{cor}\label{cor:Rukimbira2}
If the Reeb flow of $\alpha$ has exactly $n+1$ closed orbits $B_{0}$, $B_{1}$, $\ldots$, $B_{n}$, then
\begin{enumerate}
\item $H^{\bullet}(M,\F) \cong \RR[z]/(z^{n+1})$ as graded rings, where $z$ is an element of degree~$2$ which corresponds to the basic Euler class of $(M,\F)$.
\item For generic $X$ in $\mft$ the component $\Phi^{X}$ of the contact moment map satisfies $\Crit(\Phi^{X})=\bigsqcup_{i=0}^{n}B_{i}$, and each of the numbers $0$, $2$, $\ldots$, $2n$ appears exactly once as the index at a critical closed Reeb orbit.
\end{enumerate}
\end{cor}

\begin{proof}
Let $A_{\alpha}$ be the subring of $H^{\bullet}(M,\F)$ generated by the basic Euler class $[d\alpha]$. We get $\dim A_{\alpha} = n+1$, because
\begin{equation}\label{eq:S}
A_\alpha \cong \RR[z]/(z^{n+1})
\end{equation}
as graded rings where $\deg z=2$. On the other hand, $\dim H^{\bullet}(M,\F) = n+1$ by the assumption and Theorem~\ref{cor:numC}. Thus $H^{\bullet}(M,\F) = A_{\alpha}$, which proves~(1). By Proposition~\ref{prop:momentmorsebott}, the component $\Phi^{X}$ of the contact moment map is for generic $X$ a perfect basic Morse-Bott function with $\Crit(\Phi^{X}) = \bigsqcup_{i=0}^{n}B_{i}$. Proposition~\ref{prop:perfectbasic} and \eqref{eq:S} imply the latter claim of (2).
 \end{proof}

\begin{rem}
Corollary~\ref{cor:Rukimbira2} can also be shown by Remark~5.4 of~\cite{LerTol} as follows: with the help of Lemma~\ref{lem:modification} (1) and Remark~\ref{rem:transverse} we obtain a $K$-contact structure $(\alpha',g')$ so that the Reeb vector field $Z$ of $\alpha'$ generates an $S^{1}$-action. Then $M/S^{1}$ is a symplectic orbifold with a Hamiltonian function $\alpha'(Z)$. Since the closed Reeb orbits of $\alpha$ are isolated, the critical points of $\alpha'(Z)$ on $M/S^{1}$ are isolated. Then, by Remark~5.4 of~\cite{LerTol}, the number of critical points of $\alpha'(Z)$ is equal to $\dim H^{\bullet}(M/S^{1})$. On the other hand, this number is exactly $n+1$ by assumption, so $H^{\bullet}(M/S^{1})$ has to be generated by the class of the symplectic form. Again by Remark~5.4 of~\cite{LerTol}, this implies that each of the numbers $0$, $2$, $\ldots$, $2n$ appears exactly once as the index of $\alpha'(Z)$.
\end{rem}

The only-if part of the next theorem is a strengthening of Theorem~1 of~\cite{Rukimbira2}.
\begin{thm} \label{cor:n+1} If $(M,\alpha,g)$ is a compact $K$-contact $(2n+1)$-manifold whose closed Reeb orbits are isolated, then their number is exactly $n+1$ if and only if $M$ is a real cohomology sphere.
\end{thm}
\begin{proof}
Assume that the Reeb flow of $\alpha$ has exactly $n+1$ closed orbits. By Theorem~\ref{cor:odd}, the basic cohomology $H^\bullet(M,\F)$ vanishes in odd dimensions. This implies that the Gysin sequence of the isometric flow $\F$, see~\cite{Saralegui} or Equation~(7.2.1) in p.~215 of~\cite{Boyer Galicki}, splits into short exact sequences
\[
0\to H^{2k+1}(M)\to H^{2k}(M,\F)\overset{\delta}{\to} H^{2k+2}(M,\F)\to H^{2k+2}(M)\to 0,
\]
\noindent where $\delta$ is multiplication with the basic Euler class. By Corollary~\ref{cor:Rukimbira2}, $H^\bullet(M,\F)$ is, as a ring, generated by the basic Euler class, which implies that $\delta$ is an isomorphism for $k<n$. Thus, $M$ is a cohomology sphere.

Conversely assume that $M$ is a real cohomology sphere with a $K$-contact structure whose closed Reeb orbits are isolated. Then the Gysin sequence of $\F$ implies that $H^{\bullet}(M,\F)$ is generated by the basic Euler class. So Theorem~\ref{cor:numC} implies that the Reeb flow of $\alpha$ has exactly $n+1$ closed orbits.
 \end{proof}

\section{A real cohomology $7$-sphere with minimal number of closed Reeb orbits}\label{sec:Rukimbira}

Theorem~1 of Rukimbira~\cite{Rukimbira3} claims that a compact $K$-contact manifold of dimension $2n+1$ with exactly $n+1$ closed Reeb orbits is finitely covered by $S^{2n+1}$. This would be a strengthening of the only-if part of Theorem~\ref{cor:n+1} but it is not correct as the following counterexample shows.

Consider $\SO(3)$ embedded in $\SO(5)$ as $\I_2\times \SO(3)$, where $\I_2$ is the $(2\times 2)$ identity matrix. The Stiefel manifold
\[ \V_2(\RR^5)= \SO(5)/\SO(3)\]
is an example of a simply-connected real cohomology $7$-sphere which is not homeomorphic to $S^7$: The integral cohomology of $\V_2(\RR^5)$ was computed in Satz~5 of Stiefel~\cite{Stiefel} as
\[
H^{j}(\V_2(\RR^5);\ZZ) \cong \begin{cases}
\ZZ, & j=0, 7, \\
\ZZ/2\ZZ, & j = 4\\
0, & j \neq 0,4,7.
\end{cases}
\]
The homotopy exact sequence of the fiber bundle
$$
\xymatrix{
  S^{3} \cong \SO(4)/\SO(3) \ar[r]   & \SO(5)/\big(\I_{2} \times \SO(3)\big) \ar[d]\\
  &\SO(5)/\big(\I_{1} \times \SO(4)\big) \cong S^{4}}
$$
implies $\pi_{1}\big(\V_2(\RR^5)\big) = 1$.

We will construct a $K$-contact structure on $\V_2(\RR^5)$ with exactly $4$ closed Reeb orbits. First of all, $\V_2(\RR^5)$ admits a homogeneous $K$-contact structure (see, for example, Theorem~5.2 of Boyer-Galicki-Nakamaye~\cite{Boyer Galicki Nakamaye}) which we now explicitly describe. The group $\SO(5)\times \SO(2)$ acts on $\V_2(\RR^5)=\SO(5)/\SO(3)$ by $(g,h)\cdot [A]=[gAh^{-1}]$, where we identify $\SO(2)$ with $\SO(2)\times \I_3$. This action is isometric with respect to the Riemannian metric $g$ that is induced by a bi-invariant metric on $\SO(5)$. The Stiefel manifold $\V_2(\RR^5)$ is a principal $S^1$-bundle over $\Gr_2^+(\RR^5)=\SO(5)/\big(\SO(2)\times \SO(3)\big)$, the Grassmannian of oriented two-planes in $\RR^5$, where the structure group is given by the natural right action of $\SO(2)$ on $\V_2(\RR^5)=\SO(5)/\SO(3)$. Let $R$ be the fundamental vector field of this $\SO(2)$-action of unit length.
\begin{lemma} The $1$-form $\alpha:=g(R,\cdot)$ on $\V_{2}(\RR^{5})$ is an $\big(\SO(5)\times \SO(2)\big)$-invariant contact form with Reeb vector field $R$.
\end{lemma}
\begin{proof}
Let $E_{ij}$ for $1\leq i, j\leq 5$, $i\neq j$, be the $(5\times 5)$-matrix with all entries zero except $1$ at the $ij$-entry and $-1$ at the $ji$-entry. Note that $E_{ij}=-E_{ji}$. Then the Lie algebra $\so(5)$ of left-invariant vector fields is spanned by the $E_{ij}$. We choose the bi-invariant metric on $\so(5)$ such that the $E_{ij}$ are of unit length.

The vector field $R$ on $\V_2(\RR^5)$ is $\SO(5)$-invariant and therefore determined by its value at the origin, which is $E_{12}$ (up to sign). We identify the $\SO(5)$-invariant differential forms on $\V_2(\RR^5)=\SO(5)/\SO(3)$ with the $\SO(3)$-basic subcomplex of $\Lambda^\bullet \so(5)^*$. Using the Lie bracket relations of the $E_{ij}$, namely
\begin{equation}\label{eq:Eij}
[E_{ij},E_{kl}]=0, \,\,\,\, [E_{ij},E_{ik}]=E_{kj}
\end{equation}
for pairwise distinct indices $i$, $j$, $k$ and $l$, one calculates directly that $E_{12}^*$ is $\SO(3)$-basic. Therefore it corresponds to $\alpha$ under this identification. By \eqref{eq:Eij}, we compute
\[
d\alpha=\sum_{k=3}^5 E_{1k}^*\wedge E_{2k}^*,
\]
thus $\alpha\wedge (d\alpha)^3$ is a multiple of the volume form.
 \end{proof}

Let $T^3=\big(\SO(2)\times \SO(2)\times \I_1\big)\times \SO(2)$ be a maximal torus in $\big(\SO(2)\times \SO(3)\big)\times \SO(2)\subset \SO(5)\times \SO(2)$. The $T^3$-action on $\V_2(\RR^5)$ has four isolated one-dimensional orbits, which correspond to the four isolated fixed points of the $T^{2}=T^{3}/ \big(\I_{5} \times \SO(2)\big)$-action on $\Gr_2^+(\RR^5) = \V_2(\RR^5)/ \big(\I_{5} \times \SO(2)\big)$, namely the planes $\RR^2\times 0\times 0 \times 0$ and $0\times 0\times \RR^2\times 0$, with both possible orientations. Since the $T^3$-action on $\V_2(\RR^5)$ preserves $\alpha$, we get a (nonhomogeneous) $K$-contact structure $(\alpha',g')$ on $\V_{2}(\RR^{5})$ such that the closure of the Reeb flow of $\alpha'$ is equal to $T^3$ by Lemma~\ref{lem:modification} (2). Then the closed Reeb orbits of $\alpha'$ are exactly the four one-dimensional $T^3$-orbits. Thus $(\V_2(\RR^5),\alpha',g')$ satisfies the assumption of Theorem~1 of~\cite{Rukimbira3}, whereas it is not finitely covered by $S^{7}$ as we mentioned in the second paragraph of this section.

\begin{rem}
We counted all the closed Reeb orbits of $\alpha'$ explicitly above. Alternatively, we can apply Theorem~\ref{cor:n+1} to show that the number of the closed Reeb orbits of $\alpha'$ is $4$ after showing the closed Reeb orbits of $\alpha'$ are isolated, because $\V_2(\RR^5)$ is a real cohomology $7$-sphere.
\end{rem}

\begin{rem}
$K$-contact $3$-manifolds whose closed Reeb orbits are isolated are of rank $2$. Thus they are contact toric $3$-manifolds of Reeb type, which are diffeomorphic to lens spaces by Theorem~2.18 of \cite{Lerman}. Hence Theorem~1 of~\cite{Rukimbira3} is correct for dimension $3$. We do not know if there exists a counterexample of dimension $5$.
\end{rem}

\section{GKM theory for $K$-contact manifolds}\label{sec:GKM}

\subsection{Introduction}

\subsubsection{Overview of GKM theory and our result}
Consider a Cohen-Ma\-caulay action of a torus $T$ on a compact connected orientable manifold $M$, and denote by $b$ the lowest occurring orbit dimension. For every $i$, we let
\[
M_i=\{p\in M\mid \dim Tp\leq i\}.
\]
Because the action is Cohen-Macaulay, the Atiyah-Bredon sequence is exact by Theorem~6.2 of \cite{GT2009} (see Main Lemma of Bredon~\cite{Bredon} for the equivariantly formal case). We will only make use of exactness at the first terms
\begin{equation}\label{eq:ABsequence}
0\to H^{\bullet}_T(M)\to H^{\bullet}_T(M_b)\overset{\delta}{\to} H^{\bullet}_T(M_{b+1},M_b),
\end{equation}
\noindent where $\delta$ is the boundary operator of the long exact sequence in cohomology of the pair $(M_{b+1},M_b)$. Exactness of \eqref{eq:ABsequence} implies that
\begin{prop}\label{prop:ChangSkjelbred} $H^{\bullet}_T(M)\cong \ker \delta$ as graded $S(\mft^*)$-algebras.
\end{prop}
\noindent  Note that this Proposition is a version of the so-called Chang-Skjelbred Lemma (Lemma~2.3 of~\cite{ChangSkjelbred}).
As in usual GKM theory~\cite{GKM}, Proposition~\ref{prop:ChangSkjelbred} allows to give an explicit and calculable formula for the $S(\mft^*)$-algebra $H^{\bullet}_T(M)$ under additional assumptions on the action. Such a formula was first derived by Goresky-Kottwitz-MacPherson (Theorem~7.2 of~\cite{GKM}) for equivariantly formal group actions with isolated fixed points satisfying an additional assumption on the isotropy representation in the fixed points (see Definition~\ref{def:GKMconditions}). Guillemin and Holm (Theorem~1.4 in~\cite{GH}) have dropped the assumption of isolated fixed points for Hamiltonian torus actions. Here we will show how to adopt the original proof of the GKM theorem (Theorem~7.2 of~\cite{GKM}) to treat the $K$-contact and the Hamiltonian setting (both with possibly nonisolated bottom stratum) at the same time.

For symplectic manifolds with Hamiltonian torus actions satisfying the nonisolated GKM conditions (see below), the connected components of the fixed point set are diffeomorphic (Theorem~1.3 (a) of~\cite{GH}). For $K$-contact manifolds, $M_{1}$ is equal to the union $C$ of all the closed Reeb orbits. In Section~\ref{ex:C} we will present an example of a $K$-contact manifold such that the connected components of $C$ are not diffeomorphic; however we will see in Lemma~\ref{lem:orb} that their $T$-orbit spaces are diffeomorphic as orbifolds.

\subsubsection{Circle bundles over Hirzebruch surfaces}\label{ex:C}
We present an example of a $K$-contact manifold such that two connected components of the union of closed Reeb orbits are of codimension $2$ and not diffeomorphic.

We will construct $K$-contact structures on $S^{1}$-bundles over Hirzebruch surfaces explicitly, which can be obtained by the Boothby-Wang construction~\cite{Boothby Wang}. Define a $1$-form $\alpha_{k}$ on $S^{2k+1}$ by
\[
\alpha_{k} = \frac{\sqrt{-1}}{8\pi} \sum_{j=0}^{k} (z_{j}d\overline{z}_{j} - \overline{z}_{j}dz_{j})|_{S^{2k+1}},
\]
where $(z_{0}, \ldots, z_{k})$ are the standard coordinates of $\CC^{k+1}$, and $S^{2k+1}$ is embedded into $\CC^{k+1}$ as the unit sphere. Let
\[
\begin{array}{cccc}
p_{k} \colon & S^{2k+1} & \longrightarrow & \CC P^{k} \\
         & (z_{0}, \ldots ,z_{k}) & \longmapsto & [z_{0} : \ldots : z_{k}]
\end{array}
\]
be the Hopf fibration map. We get a symplectic form $\omega_{k}$ on $\CC P^{k}$ called Fubini-Study form such that
\begin{equation}\label{eq:alphaomega}
 d\alpha_{k} = p_{k}^{*}\omega_{k},
\end{equation}
which is normalized so that
\begin{equation}\label{eq:omegai}
\int_{E} \omega_{k}=1
\end{equation}
for the generator $E$ of $H_{2}(\CC P^{k};\ZZ) \cong \ZZ$ which is represented by any complex line in $\CC P^{k}$. Let $m$ be a positive integer and
\[
X = \big\{ \big([z_{0} : z_{1}], [w_{0} : w_{1} : w_{2}] \big) \in \CC P^{1} \times \CC P^{2} \, \big| \, w_{1} z_{0}^{m} = w_{2} z_{1}^{m} \big\}.
\]
Let $p = p_{1} \times p_{2}  \colon S^{3} \times S^{5} \to \CC P^{1} \times \CC P^{2}$ and
\[
Y = p^{-1}(X).
\]
Then $p|_{Y} \colon Y \to X$ is a principal $T^{2}$-bundle obtained by the restriction of the Hopf fibrations. Define a $1$-form $\alpha_{Y}$ on $Y \subset S^{3} \times S^{5}$ by
\[
\alpha_{Y} = (m\pi_{1}^{*}\alpha_{1} + \pi_{2}^{*}\alpha_{2})|_{Y}
\]
and a $2$-form $\omega$ on $X \subset \CC P^{1} \times \CC P^{2}$ by
\begin{equation}\label{eq:defomega}
\omega = (m\pi_{1}^{*}\omega_{1} + \pi_{2}^{*}\omega_{2})|_{X}
\end{equation}
where $\pi_{\ell}$ is the $\ell$-th projection. As a smooth complex submanifold of the K\"{a}hler manifold $\CC P^{1} \times \CC P^{2}$, the $4$-manifold $(X,\omega)$ is symplectic. It is called a Hirzebruch surface. By \eqref{eq:alphaomega}, we get
\begin{equation}\label{eq:alphaomega2}
d\alpha_{Y} = (p|_{Y})^{*}\omega.
\end{equation}
Let $\rho$ be the $S^{1}$-action on $Y$ defined by
\[
\lambda \cdot ( z_{0}, z_{1}, w_{0}, w_{1}, w_{2} ) = ( \lambda z_{0}, \lambda z_{1}, \lambda^{-m} w_{0}, \lambda^{-m} w_{1}, \lambda^{-m} w_{2} )
\]
for $\lambda\in S^{1} \subset \CC$ and $( z_{0}, z_{1}, w_{0}, w_{1}, w_{2} )\in Y \subset \CC^{2} \times \CC^{3}$. Let
\[
M = Y / \rho.
\]
Here $M$ is a principal $S^{1}$-bundle $p_{M} \colon M \to X$ over $X$, where the map $p_{M}$ is induced from $p|_{Y}$. Since $\alpha_{Y}$ is basic with respect to $\rho$, a $1$-form $\alpha$ on $M$ is induced from $\alpha_{Y}$. We get $d\alpha = p_{M}^{*}\omega$ by \eqref{eq:alphaomega2}. By the nondegeneracy of $\omega$, we have that $\ker d\alpha \subset TM$ is a rank $1$ subbundle tangent to the fibers of the $S^{1}$-bundle $p_{M} \colon M \to X$. Because the restriction of $\alpha$ to each fiber of the $S^{1}$-bundle $p_{M} \colon M \to X$ is nowhere vanishing, $\alpha$ is a contact form on $M$. Since $\alpha$ is invariant under the principal $S^{1}$-action on the $S^{1}$-bundle $p_{M} \colon M \to X$, the Reeb flow of $\alpha$ is the principal $S^{1}$-action on $M$. Let $\sigma$ be the $S^{1}$-action on $S^{3} \times S^{5}$ defined by
\[
\lambda \cdot ( z_{0}, z_{1}, w_{0}, w_{1}, w_{2} ) = ( z_{0}, z_{1}, \lambda w_{0}, w_{1}, w_{2} )
\]
for $\lambda$ in $S^{1} \subset \CC$ and $( z_{0}, z_{1}, w_{0}, w_{1}, w_{2} )$ in $S^{3} \times S^{5} \subset \CC^{2} \times \CC^{3}$. Then $Y$ is invariant under $\sigma$. Since $\sigma$ commutes with $\rho$, an $S^{1}$-action $\sigma_{M}$ on $M$ is induced from $\sigma$. Here $\sigma_{M}$ preserves $\alpha$, because $\sigma$ preserves $\alpha_{Y}$. Since the product of $\sigma_{M}$ and the Reeb flow of $\alpha$ defines an $\alpha$-preserving $T^{2}$-action on $M$, we have a $K$-contact structure $(\alpha', g')$ of rank $2$ on $M$ by Lemma~\ref{lem:modification} (2).

We will see that the union $C$ of closed Reeb orbits of $\alpha'$ has two connected components of codimension $2$ in $M$ which are not diffeomorphic to each other. The fixed point set of the $S^{1}$-action on $X$ induced from $\sigma$ has two connected components
\begin{align}
B_{1} & = X \cap \{w_{1} =0\} \cap \{w_{2} =0\} = \CC P^{1} \times \{[1 : 0 : 0]\}, \\
B_{2} & = X \cap \{w_{0} = 0\}.
\end{align}
Both $B_{1}$ and $B_{2}$ are complex projective lines. By \eqref{eq:omegai}, we get
\[
\int_{B_{1}} \omega = m \int_{\CC P^{1}} \omega_{1} = m.
\]
Here $B_{2}$ is the submanifold of $\CC P^{1} \times \CC P^{1} = \CC P^{1} \times \{w_{0} = 0\}$ which is the image of the map $\phi =(\phi_1,\phi_2)\colon \CC P^{1} \to \CC P^{1} \times \CC P^{1}$ defined by $\phi ([1:z]) = \big([z:1], [1:z^{m}]\big)$. Then \eqref{eq:omegai} and \eqref{eq:defomega} imply
\[
\int_{B_{2}} \omega = \int_{\CC P^1} \phi^*\omega = m \int_{\CC P^1} \omega_1 + \int_{\CC P^1} \phi_2^* \omega_2=  m+m\int_{\CC P^1} \omega_1 =  2 m.
\]
By the construction, $C$ is a principal $S^{1}$-bundle over $B_{1} \sqcup B_{2}$ whose Euler numbers are $\int_{B_{1}} \omega = m$ and $\int_{B_{2}} \omega = 2m$, respectively. Since the $S^{1}$-bundle over $S^{2}$ with Euler number $s$ is diffeomorphic to the lens space $L(s,1)$ of type $(s,1)$, we get a diffeomorphism
\[
C \cong L(m,1) \sqcup L(2m,1).
\]
Thus two connected components of $C$ are of codimension $2$ and not diffeomorphic to each other.

\subsection{A GKM type theorem for Cohen-Macaulay torus actions}

\subsubsection{Assumptions on the torus action}\label{sec:assumptionontorusaction}
Whereas our main interest lies in the $K$-contact case, we will show a GKM type theorem for general Cohen-Macaulay torus actions satisfying a certain topological condition $(*)$ (see below) to clarify a general aspect of our argument. Consider a Cohen-Macaulay action of a torus $T$ on a connected compact manifold $M$. Recall the notation $M_i=\{p\in M\mid \dim Tp\leq i\}$, and that $b$ denotes the smallest integer such that $M_{b} \neq \emptyset$.
\begin{defn}[{\cite[Definition~1.2]{GH}}]\label{def:GKMcondition}
We say that the $T$-action satisfies the {\emph{nonisolated GKM condition}} if for each $p$ in $M_b$ the weights of the isotropy representation of $T_p$ on the normal space $\nu_pM_b$ are pairwise linearly independent.
\end{defn}
\noindent   As written in~\cite{GH}, the nonisolated GKM condition is satisfied if and only if each connected component of $M_{b}$ is of codimension $2$ in the closure of a connected component of $M_{b+1}- M_b$. Note that this closure is a connected component of an isotropy manifold: For any $p$ in $M_{b+1} -M_b$ the closure of the connected component of $M_{b+1}- M_b$ containing $p$ equals the connected component
$M^{\mft_p,p}$ of
\[
M^{\mft_p}=\{q\in M\mid X^\#_q=0 \text{ for all } X\in \mft_p\}
\] containing $p$. Here $\mft_p$ is the Lie algebra of the isotropy group $T_p$ at $p$.

From now on, we assume that the Cohen-Macaulay $T$-action satisfies the nonisolated GKM condition and
\begin{itemize}
\item[$(*)$] For each closure $N$ of a connected component of $M_{b+1}- M_{b}$, there exists a $T$-invariant Morse-Bott function $h_N$ on $N$ such that $\Crit(h_N) = N \cap M_{b}$.
\end{itemize}
\noindent These two conditions have strong topological consequences as in~\cite{GH} (see Remark~\ref{rem:top} and Lemma~\ref{lem:orb}).

\begin{rem}
If $M$ is a compact $K$-contact manifold, then the action of the closure $T$ of the Reeb flow on $M$ is Cohen-Macaulay by Theorem~\ref{thm:actionisCM}. We will see below in Proposition~\ref{prop:Kcontact*} that this $T$-action always satisfies $(*)$.
\end{rem}

\begin{rem}
It is easy to see that Hamiltonian torus actions on compact symplectic manifolds satisfy $(*)$. Hamiltonian torus actions are equivariantly formal by Proposition~5.8 of \cite{Kirwan} and therefore Cohen-Macaulay by Proposition~\ref{prop:formaleqandCM}.
\end{rem}

\begin{rem}
For $K$-contact manifolds, $M_{1}$ is the union of closed Reeb orbits. For $K$-contact manifolds of rank $2$, we have $M=M_{2}$. Thus the nonisolated GKM condition is satisfied by a $K$-contact manifold $M$ of rank $2$ if and only if each connected component of $C$ is of codimension $2$ in $M$. Hence the examples of $K$-contact manifolds in Section~\ref{ex:C} satisfy the nonisolated GKM condition.
\end{rem}

\subsubsection{Definition of the GKM graph}\label{sec:GKMgraph}
The notation of Section~\ref{sec:assumptionontorusaction} will be used. The nonisolated GKM condition and $(*)$ allow us to define the GKM graph as in~\cite{GH} because of the following lemma.
\begin{lemma}\label{lem:02}
Each closure $N$ of $M_{b+1}- M_b$ contains exactly two connected components of $M_{b}$ whose Morse index is $0$ and $2$, respectively.
\end{lemma}
\begin{proof}
Let $B_{1}$, $B_{2}$, $\ldots$, $B_{k}$ be the connected components of $M_{b}$ contained in $N$. The nonisolated GKM condition implies that each $B_{i}$ is of codimension $2$ in $N$. By $(*)$, we have $\Crit(h_N) = \bigsqcup_{i=1}^{k} B_{i}$. Since $h_{N}$ is $T$-invariant, the index of $h_{N}$ at $B_{i}$ is either $0$ or $2$. Thus each $B_{i}$ is a locally maximal point set or a locally minimal point set of $h_{N}$. Since the index of all the critical manifolds of $h_{N}$ are even, both the local maximum point set and the local minimum point set of $h_{N}$ are connected. Thus we have $k=2$. The Morse index of the minimal point set is $0$, and the other is $2$.
 \end{proof}
\noindent Let us define the {\it GKM graph} $\Gamma =(\cV,\E)$.
\begin{defn}
We assign to each connected component of $M_b$ a vertex $v\in \cV$. We denote the connected component of $M_{b}$ corresponding to $v$ by $B_v$. We assign to each connected component of $M_{b+1}- M_b$ an edge $e \in \E$. We denote the closure of the connected component corresponding to $e$ by $N_e$. An edge $e\in \E$ connects $v_1$ and $v_2$ if $B_{v_1}$ and $B_{v_2}$ are the connected components of $M_b$ contained in $N_e$. We denote the two endpoints of each edge $e\in \E$ by $s(e)$ and $t(e)$.
\end{defn}

\begin{rem}\label{rem:top}
Because $M$ is connected, exactness of \eqref{eq:ABsequence} implies that $\Gamma$ is connected. In fact, the image of the injective map $H^{0}_{T}(M_{b+1}) \to H^{0}_{T}(M_{b})$ is the same as the image of $H^{0}_{T}(M) \to H^{0}_{T}(M_{b})$, which forces $H^{0}_{T}(M_{b+1})$ to be $1$-dimensional. Hence Lemma~\ref{lem:orb} below implies that for every two connected components $B_1$ and $B_2$ of $M_b$, the orbifolds $B_1/T$ and $B_2/T$ are diffeomorphic (although $B_1$ and $B_2$ are possibly not). In particular, $B_1$ and $B_2$ are of the same dimension.
\end{rem}

\subsubsection{Statement and proof of a GKM type theorem}
We will use the notation of Section~\ref{sec:assumptionontorusaction} and the GKM graph $\Gamma =(\cV,\E)$ defined in Section~\ref{sec:GKMgraph}. For $e$ in $\E$ let $h=h_{N_e}$ be a $T$-invariant Morse-Bott function as in Condition $(*)$, with minimum value $z_{\min}$ and maximum value $z_{\max}$. We assume that $B_{s(e)} =h^{-1}(z_{\min})$ and $B_{t(e)}=h^{-1}(z_{\max})$.
\begin{lemma}\label{lem:orb}
\begin{enumerate}
\item $B_{t(e)}$ (resp. $B_{s(e)}$) is a $T$-equivariant retract of $N_{e} - B_{s(e)}$ (resp. $N_{e} - B_{t(e)}$).
\item For any $z$ in $(z_{\min},z_{\max})$ there exist $T$-equivariant maps
\begin{equation*}
p_{t(e)}\colon h^{-1}(z)\to B_{t(e)}, \,\,\, p_{s(e)}\colon h^{-1}(z)\to B_{s(e)}
\end{equation*}
which define $T$-equivariant $S^{1}$-bundles and descend to orbifold diffeomorphisms of the respective orbit spaces.
\end{enumerate}
\end{lemma}
\begin{proof}
Fix a $T$-invariant metric on $N_{e}$. By Lemma~\ref{lem:02}, $N_e$ is decomposed into a disjoint union of $B_{t(e)}$ and a $T$-equivariant $D^{2}$-bundle $D^{2}(B_{s(e)}) \to B_{s(e)}$ over $B_{s(e)}$. We can $T$-equivariantly retract $N_{e} - B_{t(e)} = D^{2}(B_{s(e)})$ to $B_{s(e)}$. Considering $-h$ instead of $h$, we see that $N_{e} - B_{s(e)}$ can be $T$-equivariantly retracted to $B_{t(e)}$. This proves (1). For any $z_{1}$, $z_{2} \in (z_{\min},z_{\max})$, the gradient flow of $h$ defines an equivariant diffeomorphism from  $h^{-1}(z_{1})$ to $h^{-1}(z_{2})$. Thus level sets of $h$ are concentric $T$-equivariant $S^{1}$-subbundles of $D^{2}(B_{s(e)})$. Therefore we get $p_{s(e)}\colon h^{-1}(z)\to B_{s(e)}$ for each $z$ in $(z_{\min},z_{\max})$ by the restriction of the projection $D^{2}(B_{s(e)}) \to B_{s(e)}$. Considering $-h$ instead of $h$, we get $p_{t(e)}$. This proves (2).
 \end{proof}

By Theorem~7.1 of \cite{GT2009}, Condition $(*)$ implies that
\begin{lemma}\label{lem:NCohenMacaulay} For each $e$ in $\E$, the $T$-action on $N_e$ is Cohen-Macaulay.
\end{lemma}

Define
\[
L_e :=h^{-1}\Big(\frac{z_{\min}+z_{\max}}{2}\Big).
\]
By Lemma~\ref{lem:orb} we have $T$-equivariant $S^{1}$-bundles
\[
p_{s(e)} \colon L_{e} \longrightarrow B_{s(e)}, \,\,\, p_{t(e)} \colon L_{e} \longrightarrow B_{t(e)}.
\]
\noindent The proof of the following theorem is done by computing $\ker \delta$ in Proposition~\ref{prop:ChangSkjelbred} using the Mayer-Vietoris sequence as in the original proof of the GKM theorem (Theorem~7.2 of~\cite{GKM}).
\begin{thm}\label{thm:nonisolatedGKM} If a Cohen-Macaulay action of a torus $T$ on a connected compact manifold $M$ satisfies the nonisolated GKM condition and $(*)$, then
\[
H^{\bullet}_T(M) \cong \Big\{ (f_{v}) \in \bigoplus_{v \in \cV}  H^{\bullet}_T(B_{v}) \,\, \Big| \,\, p_{s(e)}^{*} f_{s(e)} = p_{t(e)}^{*} f_{t(e)}, \forall e \in \E \Big\}
\]
as graded $S(\mft^{*})$-algebras where $\Gamma = (\cV,\E)$ is the GKM graph of $M$.
\end{thm}

\begin{proof} By \eqref{eq:ABsequence}, we have
\[
H^\bullet_T(M)\cong \ker \big(H^\bullet_T(M_b)\to H^\bullet_T(M_{b+1},M_b)\big).
\]
Because
\[
H^\bullet_T(M_{b+1},M_b)\cong \bigoplus_{e\in \E} H^\bullet_T(N_e,B_{s(e)}\cup B_{t(e)}),
\]
a tuple $(f_{v})\in H^\bullet_T(M_b)\cong \bigoplus_{v \in \cV} H^\bullet_T(B_{v})$ is in the kernel of $\delta$ if and only if for each edge $e\in \E$ the pair $(f_{s(e)},f_{t(e)})$ is in the kernel of
\[
\delta \colon H^{\bullet}_{T}(B_{s(e)}) \oplus H^\bullet_{T}(B_{t(e)})\to H^{\bullet+1}_T(N_e,B_{s(e)}\cup B_{t(e)}).
\]

Let $e\in \E$ and consider the following diagram, in which the top row is the exact sequence of the pair $(N_e,B_{s(e)}\cup B_{t(e)})$ and the bottom row is the Mayer-Vietoris sequence of the covering $U_1=N_e - B_{t(e)}$, $U_2=N_e - B_{s(e)}$ of $N_e$:

\vspace{-9pt}

\begin{equation}\label{diag:1}\tiny{
\xymatrix{ 0\ar[r] & H^{\bullet}_{T}(N_e) \ar[r]& H^{\bullet}_{T}(B_{s(e)}) \oplus H^\bullet_{T}(B_{t(e)})\ar[r]^\delta  & H^{\bullet+1}_T(N_e,B_{s(e)}\cup B_{t(e)}) \ar[r] & 0 \\
 0 \ar[r] & H^{\bullet}_{T}(N_e) \ar[r]^<<<<<<{\theta} \ar[u]^\cong  &  H^{\bullet}_{T}(U_1) \oplus H^\bullet_T(U_2) \ar[r]^{\beta} \ar[u]^\cong_r  & H^\bullet_T(L_e) \ar[r]   & 0. } }
\end{equation}
Because the $T$-action on $N_e$ is Cohen-Macaulay by Lemma~\ref{lem:NCohenMacaulay}, the first row is exact by \eqref{eq:ABsequence}. Since $r$ is an isomorphism by Lemma~\ref{lem:orb} (1), it follows that $\theta$ is injective. Then the exactness of the Mayer-Vietoris sequence implies that $\beta$ is surjective. There exists a unique vertical isomorphism on the right that makes the diagram commute. Therefore the kernel of $\delta$ equals the kernel of $\beta\circ r^{-1}= p_{s(e)}^*~-~p_{t(e)}^*$.
 \end{proof}
\noindent See Example~\ref{ex:join} for an application of Theorem~\ref{thm:nonisolatedGKM}.

\begin{rem}
By a direct calculation, we can see that the inverse of the right vertical arrow in the diagram \eqref{diag:1} is given by the composition of the excision isomorphism and the Thom isomorphism on the normal bundle of $L_{e}$ in $N_{e}$.
\end{rem}

\begin{rem}
Note that the $T$-equivariant cohomology of $B_{v}$ is given as follows: Let $\mft_{v}$ and $\mft_{e}$ be the isotropy algebras of $B_{v}$ and $N_{e}$, respectively. Let $T'$ be an $S^{1}$-subgroup of $T$ whose Lie algebra is a complement of the isotropy algebra $\mft_{v}$ in $\mft_{e}$. Then
\[
H^{\bullet}_T(B_{v})\cong S(\mft_{v}^*)\otimes H^{\bullet}_{T'}(B_{v})\cong S(\mft_{v}^*)\otimes H^{\bullet}(B_{v}/T),
\]
\noindent as graded $S(\mft^*)$-algebras where the $S(\mft^*)$-module structure on the right hand side depends on the $T$-action on $B_{v}$. For two connected components $B_1$ and $B_2$ of $M_b$, the equivariant cohomologies $H^*_T(B_1)$ and $H^*_T(B_2)$ are hence isomorphic as rings, but not necessarily as $S(\mft^*)$-modules.
\end{rem}

\begin{rem}\label{rem:Phi}
By Lemma~\ref{lem:orb}, $p_{s(e)}$ and $p_{t(e)}$ induce orbifold diffeomorphisms $p_{s(e)}:L_e/T \to B_{s(e)}/T$ and $p_{t(e)}:L_e/T\to B_{t(e)}/T$. We denote \[
\Phi_e:= p_{t(e)}\circ p_{s(e)}^{-1}:B_{s(e)}/T\to B_{t(e)}/T.
\]
\noindent With regard to the isomorphisms $H^\bullet_T(B_{v})\cong S(\mft_{v}^*)\otimes H^\bullet(B_{v}/T)$, we may consider $f_v\in H^\bullet_T(B_{v})$ as a polynomial map $f_{v}:\mft_{v}\to H^\bullet(B_{v}/T)$. In this notation, $p_{s(e)}^* f_{s(e)}\in H^\bullet_T(L_e)\cong S(\mft_e^*)\otimes H^\bullet(L_e/T)$ is the polynomial defined by $(p_{s(e)}^* f_{s(e)})(X)=p_{s(e)}^*(f_{s(e)}(X))$ for all $X\in \mft_e$, and analogously for $p_{t(e)}^*f_{t(e)}$. Thus, the isomorphism in Theorem~\ref{thm:nonisolatedGKM} can be written as
\begin{align*}
H^{\bullet}_T(M) \cong \Big\{ (f_{v}) \in \bigoplus_{v \in \cV}  S(\mft_{v}^*)\otimes H^{\bullet}(B_{v}/T) \,\, \Big| \,\,  \begin{array}{c} f_{s(e)}(X) = \Phi_e^{*} ( f_{t(e)}(X) ) \\ {\text{ for all }} e \in \E \text{ and } X\in \mft_e\end{array} \Big\}.
\end{align*}
\noindent Comparing to the symplectic setting considered in~\cite{GH}, the maps $\Phi_e$ correspond to the ring homomorphisms $\kappa_e$ defined in Equation~(4.2) of~\cite{GH}, and the compatibility condition on the $f_{v}$ corresponds to Equation~(4.5) of~\cite{GH}. Thus our line of argumentation gives another proof of the Hamiltonian GKM theorem with nonisolated fixed points (Theorem~1.4 of~\cite{GH}) in the spirit of~\cite{GKM}.
\end{rem}

\subsection{The $K$-contact case}

We are still to show that Theorem~\ref{thm:nonisolatedGKM} is applicable in the $K$-contact setting. Let $(M,\alpha,g)$ be a connected compact $K$-contact manifold, and denote by $T$ the closure of the Reeb flow in $\Isom (M,g)$.

\begin{lemma}\label{lem:K-contactsubmfd} Let $H\subset T$ be a connected subgroup. Then each connected component of $M^H$ is a $T$-invariant contact submanifold of $M$.
\end{lemma}
\begin{proof} $M^H$ is clearly $T$-invariant, hence the Reeb vector field is everywhere tangent to $M^H$. Thus, $\ker (\alpha |_{M^H})$ is a hyperplane field on $M^H$ which coincides with $T M^H \cap \ker \alpha$. We need to show that for each $p$ in $M^H$, $(d\alpha)_p$ is nondegenerate on $\ker (\alpha |_{M^H})$. Consider the decomposition of $T_pM$ in $H$-irreducible submodules:
\[
T_pM\cong (T_pM)^H \oplus \bigoplus_\mu V_\mu\cong T_p(M^H)\oplus \bigoplus_\mu V_\mu,
\]
where $\mu$ runs over the weights of the $H$-representation: $V_\mu$ is a complex one-dimensional vector space on which the action of the Lie algebra $\mfh=\Lie(H)$ is given by $[X^\#,w]=\mu(X)Jw$.
Let $v\in T_p(M^H)$ and $w\in V_\mu$ for some $\mu$. Then for each $X$ in $\mfh$ we have
\begin{align*}
0&=(L_{X^\#}d\alpha)(v,w)\\
&=-d\alpha([X^{\#},v],w)-d\alpha(v,[X^{\#},w])=-\mu(X)d\alpha(v,Jw).
\end{align*}
It follows that $T_p(M^H)$ is orthogonal to $\bigoplus_\mu V_\mu$ with respect to $(d\alpha)_p$.

If $v$ is some element in $\big(\ker (\alpha |_{M^H})\big)_p$, then the nondegeneracy of $(d\alpha)_p$ implies that there exists $w\in (\ker \alpha)_p$ such that $(d\alpha)(v,w)\neq 0$. We may write $w=\lambda R_p+w_1+w_2$, where $\lambda\in \RR$, $w_1\in \big(\ker (\alpha |_{M^H}) \big)_p$ and $w_2\in \bigoplus_\mu V_\mu$. Then $(d\alpha)(v,w_1)\neq 0$ and it follows that  $(d\alpha)_p$ is nondegenerate on $\ker (\alpha |_{M^H})$.
 \end{proof}

If $N \subset M$ is a $T$-invariant contact submanifold, then $(N,\alpha|_{N})$ admits a compatible $K$-contact structure by Proposition~\ref{prop:Yamazaki}. Note that the effectivization of the $T$-action on $N$ is equal to the action of the torus obtained as the closure of the Reeb flow of $\alpha|_{N}$ in $\Isom (N,g|_{N})$. Modulo this effectivization, the restriction of the contact moment map of $(M,\alpha)$ to $N$ is the contact moment map of $(N,\alpha|_{N})$. Thus Proposition~\ref{prop:momentmorsebott} and Lemma~\ref{lem:K-contactsubmfd} imply:
\begin{prop}\label{prop:Kcontact*} The $T$-action on $M$ satisfies Condition $(*)$.
\end{prop}
\noindent Therefore, Theorem~\ref{thm:nonisolatedGKM} is applicable in the case of $K$-contact manifolds whose associated torus action satisfies the nonisolated GKM condition (see Example~\ref{ex:join}).

\subsection{The isolated closed orbit case}

We will see that in the isolated closed orbit case Condition $(*)$ is not needed as the manifolds $N_e$ will turn out to be orientable cohomogeneity one $T$-manifolds, whose orbit space is just a closed interval. Let $T$ be a torus, $M$ be an orientable closed $T$-manifold.
\begin{defn}\label{def:GKMconditions}
We say that the $T$-action satisfies the {\emph{isolated GKM conditions}} if $M_b$ consists of isolated $T$-orbits and the nonisolated GKM condition in Definition~\ref{def:GKMcondition} is satisfied.
\end{defn}

\begin{rem}
For equivariantly formal actions these isolated GKM conditions coincide with the usual GKM conditions, see for example Section~11.8 of Guillemin-Sternberg~\cite{GS1999}.
\end{rem}

Assume that the $T$-action is Cohen-Macaulay and that the isolated GKM conditions are satisfied. Here,  simpler than confirming Condition $(*)$, we can show Lemmas~\ref{lem:02} and~\ref{lem:orb} directly as follows:
\begin{lemma}\label{lem:coh1}
The manifold $N_{e}$ associated to each $e$ in $\E$ is a compact orientable $T$-manifold of cohomogeneity one. Thus, the well-known structure theorem of Lie group actions of cohomogeneity one implies
\begin{enumerate}
\item $N_{e}$ contains exactly two connected components of $M_{b}$,
\item $N_{e} - B_{s(e)}$ (resp. $N_{e} - B_{t(e)}$) is $T$-equivariantly retracted to $B_{t(e)}$ (resp. $B_{s(e)}$) and
\item we have $T$-equivariant maps
\begin{equation*}
p_{s(e)} \colon L_{e} \longrightarrow B_{s(e)}, \,\,\, p_{t(e)} \colon L_{e} \longrightarrow B_{t(e)}
\end{equation*}
which define $T$-equivariant $S^{1}$-bundles for every $T$-orbit $L_{e}$ of codimension one in $N_{e}$.
\end{enumerate}
\end{lemma}

\begin{proof}
By the isolated GKM conditions, each connected component of $M_{b}$ is a $T$-orbit, which is of codimension $2$ in the closure $N_{e}$ of a connected component of $M_{b+1} - M_{b}$. Thus $N_{e}$ is of dimension $b+2$. Since generic $T$-orbits in $N_{e}$ are of dimension $b+1$ in $N_{e}$, the $T$-action on $N_{e}$ is of cohomogeneity one. The orientability of $M_{b+1}$ follows from the orientability of $M$, because the normal bundle of $M_{b+1}$ can be decomposed into the sum of rank two $T$-invariant bundles with complex structures.
 \end{proof}
\noindent Lemma~\ref{lem:coh1}~(1) allows us to define the GKM graph $\Gamma = (\cV,\E)$ like in Section~\ref{sec:GKMgraph}. Since the $T$-action on $N_{e}$ is of cohomogeneity one, Example~2.(iii) in p.~829 of~\cite{GT2009} implies the Cohen-Macaulayness of the $T$-action on $N_{e}$ instead of Condition $(*)$ like in Lemma~\ref{lem:NCohenMacaulay}. By Lemma~\ref{lem:coh1}~(2) and~(3), we can apply the proof of Theorem~\ref{thm:nonisolatedGKM}. Together with Remark~\ref{rem:Phi} we obtain
\begin{cor} \label{thm:GKMgeneral} If the $T$-action on the compact connected orientable manifold $M$ is Cohen-Macaulay and satisfies the isolated GKM conditions, then
\[
H^{\bullet}_T(M)\cong \Big\{ (f_{v}) \in \bigoplus_{v \in \cV} S(\mft_{v}^*) \,\, \Big| \,\, \left.f_{s(e)}\right|_{\mft_e}=\left.f_{t(e)}\right|_{\mft_e}, \forall e \in \E \Big\}
\]
\noindent as graded $S(\mft^{*})$-algebras, where $\Gamma = (\cV,\E)$ is the GKM graph of $M$, and $\mft_{v}$ and $\mft_{e}$ are the isotropy algebras of $B_{v}$ and $N_{e}$, respectively.
\end{cor}

 By Theorem~\ref{thm:actionisCM}, Corollary~\ref{thm:GKMgeneral} and Proposition~\ref{prop:Kcontact*}, we have
\begin{cor}\label{cor:isolGKMKcontact} Let $(M,\alpha,g)$ be a compact connected $K$-contact manifold and $T$ be the closure of the Reeb flow in the isometry group. If the $T$-action satisfies the isolated GKM conditions, then
\[
H^{\bullet}_T(M)\cong\Big\{ (f_{v}) \in \bigoplus_{v \in \cV} S(\mft_{v}^*) \,\, \Big| \,\, \left.f_{s(e)}\right|_{\mft_e}=\left.f_{t(e)}\right|_{\mft_e}, \forall e \in \E \Big\}
\]
\noindent as graded $S(\mft^{*})$-algebras where $\Gamma = (\cV,\E)$ is the GKM graph of $M$, and $\mft_{v}$ and $\mft_{e}$ are the isotropy algebras of $B_{v}$ and $N_{e}$, respectively.
\end{cor}

\begin{rem}
Under the assumptions of Corollary~\ref{cor:isolGKMKcontact}, the closure of each connected component of $M_{b+1} - M_{b}$ is a $3$-dimensional contact toric manifolds of Reeb type. Then they are diffeomorphic to lens spaces by Theorem~2.18 of~\cite{Lerman}.
\end{rem}

\begin{ex}\label{ex:toricGKM}
Consider  a compact connected $K$-contact toric manifold of Reeb type $(M,\alpha,g)$ with contact moment map $\Phi \colon M \to \mft^{*}$. The image $\Delta$ of $\Phi$ is a convex compact polytope by Theorem~4.2 of \cite{Boyer Galicki 2}. Here $M$ satisfies the isolated GKM condition by Lemma~\ref{lem:weight}. Local properties of contact moment maps in Lemmas~3.10~and~3.13 of \cite{Lerman} imply that the inverse image of the union of $k$-dimensional faces of $\Delta$ under $\Phi$ is equal to $M_{k} - M_{k-1}$ where $M_{-1} = \emptyset$. Then Corollary~\ref{cor:numC} implies that the number of the vertices of $\Delta$ is equal to the dimension of $H^{\bullet}(M,\F)$ where $\F$ is the orbit foliation of the Reeb flow of $\alpha$. The GKM graph of $M$ is given by the one-skeleton of $\Delta$. By Corollary~\ref{cor:isolGKMKcontact}, $H^{\bullet}_{T}(M)$ is computed by the one-skeleton of $\Delta$ as a graded $S(\mft^{*})$-algebra.

We will give another description of the ring structure of $H^\bullet_T(M)$. Since the image of $\Phi$ is a convex compact polytope by Theorem~4.2 of~\cite{Boyer Galicki 2}, the $T$-action on $M$ is open-face-acyclic in the meaning of Definition~3 of~\cite{GT2010}. Then $H^{\bullet}_{T}(M)$ is isomorphic to the Stanley-Reisner ring of $\Delta$ as a ring by Section~10.4 of~\cite{GT2010}. A formula for integer coefficients was obtained in a different way in Luo~\cite{Shisen}.
\end{ex}

We present an example of the application of Theorem~\ref{thm:nonisolatedGKM} and Example~\ref{ex:toricGKM} to $K$-contact manifolds constructed by the fiber join construction due to Ya\nobreak mazaki~\cite{Yamazaki}. Here we apply the contact fiber bundle construction due to Lerman~\cite{Lerman 3}, which is a generalization of the fiber join construction. We follow Example 7.4 of~\cite{Lerman 3}.
\begin{ex}\label{ex:join}
Let $a=(a_{0}, \ldots, a_{n})$ be an $(n+1)$-tuple of positive real numbers. An ellipsoid
\[
E_{a} = \Big\{ (z_{0}, \ldots, z_{n}) \in \CC^{n+1} \, \Big| \, \sum_{j=0}^{n}a_{j}|z_{j}|^{2} = 1 \Big\} \cong S^{2n+1}
\]
has a $K$-contact structure given by the Euclidean metric $g_{0}$ on $\CC^{n+1}$ and a contact form $\alpha_{a} = \sqrt{-1} \sum_{j=0}^{n} (z_{j}d\overline{z}_{j} - \overline{z}_{j}dz_{j})|_{E_{a}}$. We fix a generic parameter $a$ so that the closure $T$ of the Reeb flow of $\alpha_{a}$ is given by
\begin{equation}\label{eq:fiberwiseTaction}
(w_{0}, \ldots, w_{n}) \cdot (z_{0}, \ldots, z_{n}) = (w_{0}z_{0}, \ldots, w_{n}z_{n})
\end{equation}
for $(w_{0}, \ldots, w_{n})$ in $T=T^{n+1} \subset (\CC^{\times})^{n+1}$ and $(z_{0}, \ldots, z_{n})$ in $E_{a} \subset \CC^{n+1}$.

Let $X$ be a closed surface with a volume form $\omega$ whose volume is $1$. Let $P$ be the principal $T^{n+1}$-bundle with a principal connection whose curvature form is $(\omega, \ldots, \omega)$ in $\Omega^{2}(X) \otimes \Aut(\RR^{n+1})$ where we regard $\omega$ as an element of $\Omega^{2}(X) \otimes \Aut(\RR)$. Then we have a contact form $\alpha$ on the associated $E_{a}$-bundle $M=P \times_{T^{n+1}} E_{a}$ by Theorem~4.4 of~\cite{Lerman 3}. The $T^{n+1}$-action on $P\times E_a$ by acting on the right factor descends to an effective action on $M$. By the choice of $a$ and the construction of $\alpha$ in Remark~3.8 of~\cite{Lerman 3}, we see that the closure of the Reeb flow of $\alpha$ is equal to this $T^{n+1}$-action. By Proposition~\ref{prop:Yamazaki}, we get a $K$-contact structure $(\alpha,g)$ on $M$. We see
\begin{align*}
\textstyle M_{1} & \textstyle = P \times_{T^{n+1}} \bigcup_{j=0}^{n} \{ (0, \ldots, z_{j}, \ldots, 0) \in E_a\}, \\
\textstyle M_{2} & \textstyle = P \times_{T^{n+1}} \bigcup_{0 \leq j < j' \leq n} \{ (0, \ldots, 0, z_{j}, 0, \ldots, 0, z_{j'}, 0, \ldots, 0) \in E_a \},
\end{align*}
where $M_i=\{p\in M\mid \dim Tp\leq i\}$. Each $E_{a}$-fiber $F$ in $M$ transversely intersects each connected component of $M_{1}$ and $M_{2}$. The $T$-action on $F$ is given by \eqref{eq:fiberwiseTaction}, and $(F,\alpha|_{F})$ has a structure of a $K$-contact $(2n+1)$-manifold of rank $n+1$. Then $(M,\alpha,g)$ satisfies the nonisolated GKM condition by Lemma~\ref{lem:weight}. Here $F \cap M_{1}$ is the union of $n+1$ closed Reeb orbits in $F$, and the GKM graph of $M$ is given by the one-skeleton of an $(n+1)$-simplex. Proposition~\ref{prop:Kcontact*}, Theorem~\ref{thm:nonisolatedGKM}, Remark~\ref{rem:Phi} and Example~\ref{ex:toricGKM} imply
\begin{align*}
H^{\bullet}_T(M) & \cong \ker \Big\{ \bigoplus_{0 \leq j \leq n}\!\!S(\mft_{j}^{*}) \to \!\!\!\bigoplus_{0 \leq j < j' \leq n}\!\!\!\!S(\mft_{jj'}^{*}) \Big\} \otimes H^{\bullet}(X) \\
& \cong H^{\bullet}_T(E_{a}) \otimes H^{\bullet}(X)
\end{align*}
as graded $S(\mft^{*})$-algebras where $\mft_{j}$ and $\mft_{jj'}$ are the Lie algebras of the subtori of $T$ defined by $\{ w_{j} = 0 \}$ and $\{ w_{j} = w_{j'} = 0 \}$, respectively. Note that for the second equality we apply the isolated GKM theorem (Corollary~\ref{cor:isolGKMKcontact}) to $E_{a}$ as in Example~\ref{ex:toricGKM}.
\end{ex}


\begin{thebibliography}{10}

\bibitem{allday}
Allday, C., Puppe, V.:
{ Cohomological methods in transformation groups}.
Cambridge Studies in Advanced Mathematics, 32. Cambridge University Press, Cambridge (1993)

\bibitem{Audin}
Audin, M.:
{ Torus actions on symplectic manifolds}.
Second revised edition. Progress in Mathematics, 93. Birkh\"{a}user Verlag, Basel (2004)

\bibitem{Banyaga}
Banyaga, A.:
{ A note on Weinstein's conjecture}.
Proc.~Amer.~Math.~Soc.~{\bf 109}, 855--858 (1990)

\bibitem{Banyaga Rukimbira}
Banyaga, A., Rukimbira, P.:
{ On characteristics of circle invariant presymplectic forms}.
Proc.~Amer.~Math.~Soc.~{\bf 123}, 3901--3906 (1995)

\bibitem{Boothby Wang}
Boothby, W.M., Wang, H.C.:
{ On contact manifolds}.
Ann.~of Math.~(2)~{\bf 68}, 721--734 (1958)

\bibitem{Boyer Galicki 2}
Boyer, C.P., Galicki, K.:
{ A note on toric contact geometry}.
J.~Geom.~Phys.~{\bf 35}, 288--298 (2000)

\bibitem{Boyer Galicki Nakamaye}
Boyer, C.P., Galicki, K., Nakamaye, M.:
{ Einstein metrics on rational homology $7$-spheres}.
Ann.~Inst.~Fourier (Grenoble)~{\bf 52}, 1569--1584 (2002)

\bibitem{Boyer Galicki}
Boyer, C.P., Galicki, K.:
{ Sasakian Geometry}.
Oxford Mathematical Monographs, Oxford University Press, Oxford (2007)

\bibitem{Bredon}
Bredon, G.E.:
{ The free part of a torus action and related numerical equalities}.
Duke Math.~J.~{\bf 41}, 843--854 (1974)

\bibitem{Carriere}
Carri\`{e}re, Y.:
{ Flots riemanniens}.
Transversal structure of foliations (Toulouse, 1982). Ast\'{e}risque {\bf 116}, 31--52 (1984)

\bibitem{ChangSkjelbred}
Chang, T., Skjelbred, T.:
{ The topological Schur lemma and related results}.
Ann.~of Math.~(2) {\bf 100}, 307--321 (1974)

\bibitem{FrPu}
Franz, M., Puppe, V.:
{ Exact sequences for equivariantly formal spaces}.
C.~R.~Math.~Acad.~Sci.~Soc.~R.~Can.~{\bf 33}, 1--10 (2011)

\bibitem{GT2009}
Goertsches, O., T\"oben, D.:
{ Torus actions whose equivariant cohomology is Cohen-Macaulay}.
J.~Topology~{\bf 3}, 819--846 (2010)

\bibitem{GT2010}
Goertsches, O., T\"oben, D.:
{ Equivariant basic cohomology of Riemannian foliations}.
Preprint (2010), {\texttt{arXiv:1004.1043}.}

\bibitem{GKM}
Goresky, M., Kottwitz, R., MacPherson, R.:
{ Equivariant cohomology, Koszul duality, and the localization theorem}.
Invent.~Math.~{\bf 131}, 25--83 (1998)

\bibitem{GH}
Guillemin, V., Holm, T.S.:
{ GKM Theory for Torus Actions with Nonisolated Fixed Points}.
Int.~Math.~Res.~Notices~{\bf 40}, 2105--2124 (2004)

\bibitem{GS1999}
Guillemin, V., Sternberg, S.:
{ Supersymmetry and Equivariant de Rham Theory}.
Springer-Verlag, Berlin (1999)

\bibitem{Hsiang}
Hsiang, W.-Y.:
{ Cohomology theory of topological transformation groups}.
Ergebnisse der Mathematik und ihrer Grenzgebiete, Band~85, Springer-Verlag, New York-Heidelberg (1975)

\bibitem{Kirwan}
Kirwan, F.:
{ Cohomology of quotients in symplectic and algebraic geometry}.
Mathematical Notes,~31. Princeton University Press, Princeton (1984)

\bibitem{Lerman}
Lerman, E.:
{ Contact Toric Manifolds}.
J.~Symplectic~Geom. {\bf 1}, 785--828 (2002)

\bibitem{Lerman 3}
Lerman, E.:
{ Contact fiber bundles}.
J.~Geom.~Phys.~{\bf 49}, 52--66 (2004)

\bibitem{LerTol}
Lerman, E., Tolman, S.:
{ Hamiltonian torus actions on symplectic orbifolds and toric varieties}.
Trans.~Amer.~Math.~Soc.~{\bf 349}, 4201--4230 (1997)

\bibitem{Shisen}
Luo, S:
{ Cohomology rings of good contact toric manifolds}.
Preprint (2010), {\texttt{arXiv:1012.2146}.}

\bibitem{Molino}
Molino, P.:
{ Riemannian foliations}.
with appendices by G.~Cairns, Y.~Carri\`ere, \'E.~Ghys, E.~Salem and V.~Sergiescu, Birkh\"auser Boston Inc., Boston (1988)

\bibitem{MolSer}
Molino, P., Sergiescu, V.:
{\it Deux remarques sur les flots riemanniens}.
Manuscripta Math.~{\bf 51}, 145--161 (1985)

\bibitem{Myers Steenrod}
Myers, S.B., Steenrod, N.E.:
{ The group of isometries of a Riemannian manifold}.
Ann.~of Math.~(2)~{\bf 40}, 400--416 (1939)

\bibitem{Nozawa}
Nozawa, H.:
{ Five dimensional $K$-contact manifolds of rank $2$}.
Doctor Thesis in the University of Tokyo (2009), {\texttt{arXiv:0907.0208}.}

\bibitem{Rukimbira4}
Rukimbira, P.:
{ The dimension of leaf closures of $K$-contact flows}.
Ann.~Global Anal.~Geom. {\bf 12}, 103--108 (1994)

\bibitem{Rukimbira}
Rukimbira, P.:
{ Topology and closed characteristics of $K$-contact manifolds}.
Bull.~Belg.~Math.~Soc.~Simon Stevin~{\bf 2}, 349--356 (1995)

\bibitem{Rukimbira2}
Rukimbira, P.:
{ On $K$-contact manifolds with minimal number of closed characteristics}.
Proc.~Amer.~Math.~Soc.~{\bf 127}, 3345--3351 (1999)

\bibitem{Rukimbira3}
Rukimbira, P.:
{ Correction to: ``Spherical rigidity via contact dynamics''}.
Bull.~Belg.~Math.~Soc.~Simon Stevin~{\bf 8}, 147--153 (2001)

\bibitem{Saralegui}
Saralegui, M.:
{ The Euler class for flows of isometries}.
In: Cordero, L.A. (eds.) Res.~Notes in Math., 131, pp. 220--227. Pitman, Boston, MA, (1985)



\bibitem{Stiefel}
Stiefel, E.:
{ Richtungsfelder und Fernparallelismus in $n$-dimensionalen
Mannigfaltigkeiten}.
Comm.~Math.~Helv.~{\bf 8}, 305--353 (1935)

\bibitem{Takahashi}
Takahashi, T.:
{ Deformations of Sasakian structures and its application to the Brieskorn manifolds}.
T\^{o}hoku Math.~J.~(2)~{\bf 30}, 37--43 (1978)

\bibitem{Yamazaki}
Yamazaki, T.:
{ A construction of $K$-contact manifolds by a fiber join}.
T\^{o}hoku Math.~J.~(2)~{\bf 51}, 433--446 (1999)

\end{thebibliography}
\end{document}